\documentclass[11pt]{article}
\usepackage{geometry}                % See geometry.pdf to learn the layout options. There are lots.
\usepackage{graphicx}
\usepackage{amssymb}
\usepackage{chemarrow}
\usepackage{amsmath}
\usepackage{mathrsfs}
\usepackage{mathtools}
\usepackage{epstopdf}
\usepackage[title]{appendix}
\newcommand{\rd}{\textrm{d}}
\newcommand{\MI}{\text{MI}}
\newcommand{\dash}{\hspace{-0.04cm}-\hspace{-0.04cm}}
\usepackage{amsthm}
\newtheorem{theorem}{Theorem}
\newtheorem{proposition}{Proposition}

\newtheorem{lemma}{Lemma}[]
\newtheorem{remark}{Remark}[]
\newtheorem{definition}{Definition}
\usepackage{multirow}
\usepackage{setspace}
\usepackage{authblk}

\usepackage[plain]{algorithm2e}
\usepackage{tikz}

\usepackage[all,cmtip]{xy}
\usepackage{lineno}
\usepackage{bm}
\title{Mathematical Representation of Clausius’ and Kelvin’s Statements of the Second Law and Irreversibility}
\author[1,2,*]{Yue Wang}
\author[1]{Hong Qian}
\affil[1]{Department of Applied Mathematics, University of Washington, Seattle, Washington, U.S.A.}
\affil[2]{Institut des Hautes \'Etudes Scientifiques, Bures-sur-Yvette, Essonne, France}
\affil[*]{E-mail address: yuewang@ihes.fr (Y. W.). ORCID: 0000-0001-5918-7525.  \newline\newline \emph{In memory of LI Wenliang (1986–2020), a brave physician}}
\date{}                                           % Activate to display a given date or no date
\begin{document}
\maketitle

\begin{abstract}
We provide a stochastic mathematical representation for 
Clausius' and Kelvin-Planck's statements of the Second
Law of Thermodynamics in terms of the entropy productions of
a finite, compact driven Markov system and its {\it lift}.  A 
surjective map is rigorously established through the lift 
when the state space is either a discrete graph or a continuous 
$n$-dimensional torus $\mathbb{T}^n$. The corresponding lifted processes
have detailed balance thus a natural potential function
but no stationary probability.  We show that in the 
long-time limit the entropy production of the finite driven system 
precisely equals the potential energy decrease in the lifted system.  
This theorem provides a dynamic foundation for the two
equivalent statements of Second Law of Thermodynamics, 
{\it \`{a} la} Kelvin's and Clausius'.  It suggests a modernized, combined statement:  ``A mesoscopic engine that works in completing irreversible internal cycles statistically has necessarily an external effect that lowering a weight accompanied by passing heat from a warmer to a colder 
body.''
	
\end{abstract}

\smallskip
\noindent \textbf{Keywords.} 

\noindent Entropy production; Markov chain; Diffusion process; Second Law of Thermodynamics.

\section{Introduction}

There is a growing awareness toward a slow shifting in 
the foundation of the Second Law of Thermodynamics, from a
macroscopic postulate concerning heat as a form of random
mechanical motion \cite{pauli} to a derivable mathematical discovery 
based on the stochastic dynamics of mesoscopic 
systems \cite{jarzynski,seifert}.  
The first significant attempt 
in this direction was carried out by L. Boltzmann through the
equation that now bears his name and the H-theorem it derives.
The theory is applicable to gas dynamics; a fundamental assumption underlying the classic work is a {\em stosszahlansatz}
\cite{herzfeld,dorfman}.  Rigorous mathematical breakthrough on 
Boltzmann's equation only became available very recently 
\cite{villani}. In 1950s, Bergmann and Lebowitz set up a general
stochastic theory for closed as well as open mechanical systems 
that are consistent with Hamiltonian dynamics, and easily obtained 
an H-theorem like result \cite{lebowitz-bergmann}.  It becomes 
increasingly clear in recent years that a stochastic description 
of the Nature is a very powerful mathematical
representation \cite{qian-bc-1,qian-bc-2}.  

For dynamics that can be represented in terms of a Markov
process, a rather coherent system of {\em mesoscopic
	theory of entropy productions} has emerged.  
See many papers \cite{ge-qian-10,vandenbroek-esposito,esposito-vandenbroek,NESS1,qian-jmp} and references cited within.  
More recently, when this theory is applied to general 
chemical reaction systems represented by stochastic kinetics of elementary reactions, a result that is 
consistent with and further generalizes 
Gibbsian macroscopic chemical thermodynamics 
has been obtained, as a mathematical limit by merely allowing 
the molecular numbers to be infinite
\cite{ge-qian-16-1,ge-qian-16-2}.  In particular it is possible
to show, rigorously for the first time, that for each and every 
elementary, reversible reaction with instantaneous 
forward and backward rates $R^+$ and $R^-$, the 
macroscopic entropy production rate is 
$(R^+-R^-)\log (R^+/R^-)$
\cite{ross,luojl,beard-qian-plos-1,Levchenko}.

In a Markov description of a driven system, irreversible
kinetic {\em cycles} have been identified as fundamental to 
entropy production \cite{hill-book,JQQ,altaner,qkkb}.  From the
standpoint of an observer who simultaneously
follows the system's internal stochastic dynamics as well as 
the external driving mechanism, there is a dissipation associated
with a ``falling weight'' \cite{Lieb-Yngvason}, e.g., a
potential energy drop as a spontaneous, relaxation process.  
These two different perspectives fittingly echo 
the two fundamental statements
of the Second Law of Thermodynamics, from Kelvin and Planck
and from Clausius respectively \cite{planck}:\begin{quotation}
	``{\it It is impossible to construct an engine which will work in a 
		complete cycle, and produce no effect except the raising of 
		a weight and the cooling of a heat-reservoir.}''
\end{quotation}

\begin{quotation}
	``{\it Heat can never pass from a colder to a warmer body without some other change, connected therewith, occurring at the same time.}'' 
\end{quotation}
In the present paper, we shall show that in the setting of
non-detailed balanced Markov processes, either with a finite state space or 
on a continuous $n$-torus with local potential, counting irreversible
kinetic cycles constitutes a {\em lift} of the Markov processes
into an infinite-state Markov process or diffusion on $\mathbb{R}^n$,
respectively. {Here a process is detailed balanced if it has a ``detailed balanced stationary distribution/measure'', starting from which the local flux is zero everywhere. In other words, the process and its time inverse are equivalent. A cycle is reversible if the process restricted on this cycle is detailed balanced. A process is detailed balanced if and only if all cycles are reversible.}
The lifted Markov process satisfies detailed balance; it possesses many
different stationary measures.  However, detailed balance guarantees a 
natural potential function and thus a corresponding Gibbsian 
invariant measure.  This ``no-flux'' Gibbs measure 
is non-normalizable; its unbounded potential function provides 
a rigorous notion of an ``internal energy'' function $\varphi_x$.  
This is a new result; all previous work assumed 
the existence of a unique stationary probability measure as $t\to\infty$.

More specific mathematically,
we shall show that, in the limit of $t\to\infty$, the positive stationary 
entropy production rate in the original Markov process is precisely the 
change in mean internal energy, $E\equiv\mathbb{E}[\varphi]$, of the
free energy dissipation $\dot{F}=\dot{E}-\dot{S}$ in the lifted system.  
The energetic part of the free energy grows linearly with $t$; the entropic part of 
the free energy dissipation vanishes as $t\to\infty$
in Ces\`{a}ro's sense $\dot{S}\equiv S(t)/t\to 0$.  Surprisingly, 
in rigorous mathematics, we have not been able to show the
stronger assertion that $\rd S(t)/\rd t \to 0$ in general except 
some very special cases.

\begin{figure}[t]
	$
	\xymatrix{
		&&S \ar@/_1pc/[rd]&&&&&&\\
		&P\ar@/^1pc/[rd]&E\ar[ddd]_{k_{3}}\ar[rrr]^{k_1^0}&&&ES\ar@/^1pc/[lll]^{k_{-1}}\ar@/^1pc/[lllddd]^{k_{2}}&&&\\
		&&&&&&&&\\
		&&&&&&&&\\
		&&EP\ar[rrruuu]^{k_{-2}}\ar@/_1pc/[uuu]_{k_{-3}^0}&&&&&&\\
		\cdots\ar@/^/[r]&E\ar@/^/[r]\ar@/^/[l]&ES \ar@/^/[r]\ar@/^/[l]&EP\ar@/^/[r]\ar@/^/[l]&E\ar@/^/[r]\ar@/^/[l]&ES\ar@/^/[r]\ar@/^/[l]&EP\ar@/^/[r]\ar@/^/[l]&E\ar@/^/[r]\ar@/^/[l]&\cdots\ar@/^/[l]\\
		&0,0&1,0&1,0&1,1&2,1&2,1&2,2&
	} 
	$\\
	\caption[]{(Up) A chemical reaction with three states: E (enzyme) binds to an incoming S (substrate) to form ES, which transforms to EP, and decomposes to E and P (product). After a clockwise cycle, this reaction turns an S into a P. (Middle) We can expand this reaction cycle into a line. (Down) The number of S consumed and P produced. For example, $E(1,1)$ means that the system has consumed one S and produced one P. When the system starts from $E(0,0)$, ends at $E(2,2)$, we can see that the system turns two S into P.}
	\label{eps}
\end{figure}

{Consider an enzyme (E) molecule catalyze the reaction from a substrate (S) to a product (P)} \cite{ge2012stochastic}:
$$E+S\xrightleftharpoons[k_{-1}]{\,k_1^0\,} ES \xrightleftharpoons[k_{-2}]{\,k_2\,} EP \xrightleftharpoons[k_{-3}^0]{\,k_3\,} E+P.$$
{Assume there is only one E, then the system can be regarded as a Markov chain with a cycle. Maintain the concentrations of S and P ($c_S,c_P$) at constant levels such that $c_Sk_1^0k_2k_3>c_Pk_{-1}k_{-2}k_{-3}^0$. Then the system tends to converse S to P. After a full cycle $E\to ES\to EP\to E$, the system returns to initial state, while one S turns into P. Along this cycle, free energy dissipates (with entropy production). To reflect the change of thermodynamic quantities, we can open the cycle through lifting} (Fig. \ref{eps}). {In the lifting, the process does not return to its initial state after a cycle, but moves right. Therefore we can calculate the number of P made, and the associated thermodynamic quantities, solely from the starting and ending points.}

{Although not all processes are Markov, we can use Markov processes to approximate time-homogeneous non-Markov processes through enlarging state space. For example, if $X_{n+1}$ depends on $X_n$ and $X_{n-1}$, then $\{X_n\}$ is not Markov, but $\{(X_{n-1},X_n)\}$ is Markov.}

To statistical thermodynamics, our mathematical theory
rigorously establishes an equivalence between the two famous
statements concerning the Second Law: A cyclic 
view of dissipation \cite{hill-book,bennett}
and a non-stationary relaxation view of irreversibility.  In fact,
there is a continuous surjective map between the 
trajectories before and after the lifting: In the 
long-time behavior, the cycle
completion and entropy production in the former is 
precisely represented by the potential change in 
the latter.  Dissipation in the former is due to 
indistinguishability of the locally equivalent states
in the latter.

In textbook teaching of the Second Law, showing an 
equivalence of the two statements has a prominent place in 
the logic of the thermodynamics.  In terms of the stochastic thermodynamics based on Markov representations, this paper 
provides novel representations of the two statements:  
One is based on a finite state space and irreversible 
cycles within, and the other is based on the lift of the former 
into an infinite state space on which a potential function 
exists due to detailed balance.  We show the stationary
house-keeping heat in the former is indeed the free energy dissipation of the latter!  In other words, a nonequilibrium 
stationary state can be understood as a spontaneous 
relaxation process on a free energy landscape.  Our theorem
thus provides a stochastic dynamic underpinning for the 
two celebrated verbal statements.

{Finally but not the least, our mathematical result again suggests that the notions of entropy, entropy production, and now together with the new understanding of the Second Law in terms of the Clausius-Kelvin principle, are no longer merely a part of thermal physics. They are valid concepts for any systems and processes that can be modeled by a stochastic Markov process. In fact, one expects these mathematical concepts provide deeper understanding of the nature of ``force'' and ``energy'' in a wide range of problems outside traditional physics.}

The paper is structured as follows:  In Sec. \ref{sec-2},
we prove an embedding theorem that establishes a
minimal lift of a continuous time, finite state Markov chain 
to a detailed balanced process with a proper potential 
function.   Then
in Sec. \ref{sec-3}, we prove the theorem that,
in the limit of $t\to\infty$, equating the entropy 
production rate $\bar{e_p}(t)$ of the finite system 
with the Ces\`{a}ro limit of $e_p(t)$ from the lifted 
system.  For lifted, detailed balanced processes,
$e_p(t)$ can be expressed as $-\rd F/\rd t$
where free energy
$F(t)=\text{D}_{\text{KL}}\big(p(t),\mu\big)$ 
is the relative entropy of $p(t)$ with respect to the
Gibbs measure $\mu=e^{-\varphi}$.  We show as 
$t\to\infty$, $F(t)=E(t)-S(t)$ has a linearly decreasing 
$E(t)$ and a sublinear $S(t)$ controlled by $\log t$.  Therefore
in the long time limit $\bar{e_p}$ equals to $-\dot{E}$. Sec. \ref{sec-n-torus}
is a mathematical generalization of Sec. \ref{sec-3} to 
diffusion processes on $n$-torus and their lifting to $\mathbb{R}^n$.  The paper concludes with Sec. \ref{sec-6}.

\section{The Lifted Markov Chain and Its Infinite State Space}
\label{sec-2}

In this section, we study the lifting of a continuous time,
finite state Markov chain with transition rate matrix $\{q_{ij}\}$.  First, we need some prerequisites in different fields.

\subsection{Prerequisites}
\label{sec:II-A}
{Readers may refer to books on algebraic topology} \cite{munkres1974topology}, graph theory \cite{G173}, algebraic graph theory  \cite{godsil2013algebraic}, Markov chains \cite{norris1998markov}, diffusion processes \cite{oksendal2003stochastic} and thermodynamic quantities in stochastic processes \cite{JQQ} for other prerequisites that are not covered in this section.
\subsubsection{Algebraic Topology}
\begin{definition}
	{Let $X$ be a topological space. A covering space of $X$ is a topological space $C$ together with a continuous surjective map $p: C\to X$, such that for every $x\in X$, there is an open neighborhood $U$ of $x$, such that $p^{-1}(U)$ is a union of disjoint open sets in $C$, each of which is mapped homeomorphically onto $U$ by $p$. A path in $X$ can be uniquely lifted to $C$ with a given starting point.}
\end{definition}
\begin{definition}
	{A covering space is a universal covering space if it is simply connected. General space, such as connected graph or $n$ dimensional torus, has universal covering space. If exists, universal covering space is unique.}
\end{definition}

{The universal covering space of $1$-dimensional torus $\mathbb{S}^1$ is $\mathbb{R}$. The universal covering space of $n$-dimensional torus $\mathbb{T}^n$ is $\mathbb{R}^n$} \cite{munkres1974topology}.

\subsubsection{Graph Theory}

Consider an undirected connected simple finite graph $(V,E)$ with vertex set $V$ and edge set $E$.
\begin{definition}
	The \emph{first Betti number} of $(V,E)$ is $b(V,E)=|E|-|V|+1$. It is the number of ``independent cycles'' \cite{G173}.
\end{definition}
\begin{definition}
	A subgraph of $(V,E)$ is called a \emph{spanning tree} if it is a tree, and contains all the vertices of $(V,E)$. 
\end{definition}
Spanning tree always exist, and has $b(V,E)$ less edges than $(V,E)$, since the first Betti number of a tree is zero \cite{G173}.

\begin{definition}
	A covering graph of $(V,E)$ is a graph $(V',E')$ together with a covering map $p: V'\to V$, which is a surjection and a local isomorphism: if $p(v')=v$,then neighbors of $v'$ are mapped by $p$ bijectively to neighbors of $v$.
\end{definition}
{Embed $(V,E)$ into $\mathbb{R}^n$ such that edges do not intersect. Then the natural topology of $\mathbb{R}^n$ induces a topology on $(V,E)$. Under this topology, a covering graph is just a covering space.} A trajectory in $(V',E')$ can be folded into $(V,E)$ via $p$. A trajectory in $(V,E)$ starting from $v$ can be lifted to a trajectory in $(V',E')$ starting from any $v'$ with $p(v')=v$ \cite{godsil2013algebraic}.

\begin{definition}
	A covering graph is universal if it has no cycle. Universal covering graph exists, and is unique up to isomorphism \cite{godsil2013algebraic}.
\end{definition}

\subsubsection{Potential of Markov chain}
Here and in the following, we always require that $q_{ij}>0$ if and only if $q_{ji}>0$. 
\begin{definition}
	For a Markov chain, if there is a function on state space, $f(i)$, such that for any two adjacent states $i,j$, $f(j)-f(i)=\log(q_{ji}/q_{ij})$, then $f$ is called a global potential. If exists, it is unique up to a constant.
\end{definition}
\begin{definition}
	A global potential is proper if different states have different potentials. 
\end{definition}
In this paper, potential is always calculated symbolically (as a function of $\{q_{ij}\}$), not numerically.

For a Markov chain, we define the potential gain of a trajectory $i_1,\cdots,i_k$ as 
$$\sum_{j=1}^{k-1}\log\frac{q_{i_{j+1} i_j}}{q_{i_j i_{j+1}}}.$$
If global potential exists, then the potential gain is the difference of potential on the two end states. 

If the potential gain of any closed trajectory is zero, then there is a global potential. If in addition, any trajectory with zero potential gain is closed, then the global potential is proper.

\subsection{Embedding a Markov chain into an $n$-torus}

\subsubsection{Motivation and Results}

Consider a continuous time irreducible Markov chain $M$ with finite states and transition rate matrix $Q=\{q_{ij}\}$. Then we can define a graph $(V,E)$. Vertices are states of this Markov chain, and edges are possible transitions.  

When a trajectory finishes a cycle, the potential gain is not zero in general, although it returns to its starting point. Thus the current state cannot properly reflect the history of potential change. We can try to expand such cycles with non-zero potential gain.

To illustrate our idea, consider a $3$-state Markov chain, with one cycle $1\dash 2\dash 3\dash 1$. We can embed the corresponding graph into $\mathbb{S}^1$, and then lift it to $\mathbb{R}$. Now it is $\cdots\dash 1\dash 2\dash 3\dash 1\dash 2\dash 3\dash 1\dash\cdots$ (Fig. \ref{1dlift}). There exists a proper global potential. As long as we know the ends of a trajectory, we know its potential gain.

\begin{figure}[t]
	$
	\xymatrix{
		2\ar@{-}[d]& &&&2\ar@{-}[d] &2\ar@{-}[d]&2\ar@{-}[d]&2\ar@{-}[d]&&\\
		3&\ar@{=>}[r]&&\cdots\ar@{-}[rrd] &3\ar@{-}[rrd]&3\ar@{-}[rrd]&3\ar@{-}[rrd]&3\ar@{-}[rrd]&&\\
		&1\ar@{-}[lu]\ar@{-}[luu]&& & &1\ar@{-}[luu]&1\ar@{-}[luu]&1\ar@{-}[luu]&1\ar@{-}[luu]&\cdots\\
		\ar@{=>}[r]&&\cdots\ar@{-}[r]&1\ar@{-}[r]&2\ar@{-}[r]&3\ar@{-}[r]&1\ar@{-}[r]&2\ar@{-}[r]&3\ar@{-}[r]&\cdots
	} 
	$\\
	\caption[]{The lifting of a $3$-state Markov chain. This chain is embedded into $\mathbb{S}^1$. Then $\mathbb{S}^1$ and the embedded chain is lifted to $\mathbb{R}$, namely the universal cover of $\mathbb{S}^1$. The embedded chain is also lifted to an infinite-state chain with global potential.}
	\label{1dlift}
\end{figure}

Now the goal is to find a new Markov chain, which is locally isomorphic to $M$ (a covering graph), and has a proper global potential. 

We require that the global potential is proper, since there is no need to distinguish between two states with the same potential. Also, properness guarantees that the cover is unique. When and only when there is a proper global potential, a trajectory has zero potential gain if and only it is closed. 

To ensure the existence of global potential, we can consider the universal covering graph of $M$. Cover guarantees local isomorphism, and universality implies there is no cycle, which means the existence of a global potential. However, the global potential on the universal cover might not be proper. A natural idea is to glue together states with the same potential. 

Consider a Markov chain with a single state and $n$ cycles. Its universal cover is the Cayley graph of free group $F_n$. In this universal cover, states have the same potential if and only if the abelianization of their corresponding words are the same. If states with the same potential are glued together, then we have the Cayley graph of $\mathbb{Z}^n$, which is the abelianization of $F_n$. For a general Markov chain, the covering graph with global potential should be in $\mathbb{R}^n$, as the repetition of a unit on $\mathbb{Z}^n$ lattice. Since $\mathbb{R}^n/\mathbb{Z}^n=\mathbb{T}^n$, the $n$-torus, we expect that the unit should be in $\mathbb{T}^n$. This is the intuition of the following theorems: why the Markov chain is embedded into $\mathbb{T}^n$, and why the cover with proper global potential is in $\mathbb{Z}^n$ lattice. Readers can refer to a standard algebraic topology textbook \cite{munkres1974topology} for notions used in this paragraph.

\begin{theorem}[torus version]
	A finite Markov chain $M$ with first Betti number $n$ can be embedded into $n$-torus $\mathbb{T}^n$, such that any closed trajectory has zero potential gain if and only if it is homotopy trivial. Here homotopy trivial means that the trajectory could continuously transform into a single point, as a closed curve in $\mathbb{T}^n$. (For $n=1$, only the cycle can be embedded.)
	\label{T1}
\end{theorem}

\begin{theorem}[lifted version]
	When $\mathbb{T}^n$ with the embedded Markov chain as in Theorem \ref{T1} is lifted to its universal cover $\mathbb{R}^n$, the correspondingly lifted Markov chain $L$ has a proper global potential.
	\label{T2}
\end{theorem}

Theorem \ref{T1} connects two seemingly irrelevant properties through a coincidence: a closed trajectory in $\mathbb{T}^n$ is homotopy trivial, if and only if its lifting in $\mathbb{R}^n$ is closed; a closed trajectory in $M$ has zero potential gain, if and only if its lifting in $L$ (which has a proper global potential) is closed. Here is the essence of this coincidence: Regard $\mathbb{T}^n$ as the Cartesian product of $n$ one-dimensional tori $\mathbb{S}^1$, then a closed trajectory in $\mathbb{T}^n$ is homotopy trivial if and only if its projection on each $\mathbb{S}^1$ has winding number zero. The Markov chain with first Betti number $n$ has $n$ independent cycles \cite{Pol}, and a closed trajectory has zero potential gain if and only if each independent cycle has winding number zero \cite{JQQ}. The embedding in Theorem \ref{T1} in fact embeds each independent cycle into one $\mathbb{S}^1$, therefore unifies the winding number of topological cycle and graph theoretical cycle.

\begin{proposition}
	$L$ is the minimal covering graph of $M$ with global potential, in the sense that a covering graph $C$ of $M$ has global potential if and only if it is a covering graph of $L$. Besides, $L$ is the only covering graph of $M$ with proper global potential.
	\label{P1}
\end{proposition}

\subsubsection{Proofs}
\begin{proof}[Theorem \ref{T2}]
	We will use the example in Sec. \ref{sec:II-A} as illustrations.
	
	Regard $\mathbb{T}^n$ as the unit hypercube $[0,1]^n$ with opposite hypersurfaces glued together. For $M$ with first Betti number $b(V,E)=n$, choose a spanning tree, and embed it into $\mathbb{T}^n$. For each edge of $M$ that is not in the spanning tree (we have $n$ of such special edges), assign a pair of opposite hypersurfaces to it. Draw this edge in $\mathbb{T}^n$ while crossing the corresponding hypersurface once. Now we have embedded $M$ into $\mathbb{T}^n$ (Fig. \ref{embed}).

	\begin{figure}[t]
		$
		\xymatrix{
			&&&&&&(3)\ar@{=}[ddr]&&&\\
			&&&&&+\ar@{-}[dddd]\ar@{-}[rrr]&&&+\ar@{-}[dddd]&\\
			&4 && 4&&&&4&&\\
			&1\ar@{-}[u]\ar@{-}[ld]\ar@{-}[d]& &1\ar@{-}[u]\ar@{-}[ld]\ar@{-}[d]&&&&1\ar@{-}[u]\ar@{-}[ld]\ar@{-}[d]&&\\
			3\ar@{=}[uur]&2\ar@{--}[l]& 3& 2&(2)\ar@{--}[rr]&&3&2&&(3)\ar@{--}[ll]\\
			&&&&&+&&&+\ar@{-}[lll]&\\
			&&&&&&&(4)\ar@{=}[uul]&&
		} 
		$\\
		\caption[]{In this example, the Markov chain (left) has first Betti number $b(V,E)=2$, therefore we can choose a spanning tree (middle), which corresponds to two special edges $2- - 3$ and $3=\hspace{-0.05cm}=4$. Embed the spanning tree into $\mathbb{T}^2$, which is a square in $\mathbb{R}^2$ with opposite boundaries glued. Assign special edge $2- - 3$ to vertical boundaries, and $3=\hspace{-0.05cm}=4$ to horizontal boundaries. Then connect $2$ and $3$ across the vertical boundaries, connect $3$ and $4$ across the horizontal boundaries. Now we have embedded the Markov chain into $\mathbb{T}^2$ (right).}
		\label{embed}
	\end{figure}

	\begin{figure}[t]
		$\xymatrix{
			&\ar@{-}[dddddddddddddd]&\cdots&&\ar@{-}[dddddddddddddd]&\cdots&&\ar@{-}[dddddddddddddd]&\cdots&&\ar@{-}[dddddddddddddd]&\\
			\ar@{-}[rrrrrrrrrrr]&&&&&&&&&&&\\
			&&&4\ar@{=}[uul]&&&4\ar@{=}[uul]&&&4\ar@{=}[uul]&&\\
			&&&1\ar@{-}[u]\ar@{-}[ld]\ar@{-}[d]&&&1\ar@{-}[u]\ar@{-}[ld]\ar@{-}[d]&&&1\ar@{-}[u]\ar@{-}[ld]\ar@{-}[d]&&\\
			\cdots&&3\ar@{--}[ll]&2&&3\ar@{--}[ll]&2&&3\ar@{--}[ll]&2&&\cdots\ar@{--}[ll]\\
			\ar@{-}[rrrrrrrrrrr]&&&&&&&&&&&\\
			&&&4\ar@{=}[uul]&&&4\ar@{=}[uul]&&&4\ar@{=}[uul]&&\\
			&&&1\ar@{-}[u]\ar@{-}[ld]\ar@{-}[d]&&&1\ar@{-}[u]\ar@{-}[ld]\ar@{-}[d]&&&1\ar@{-}[u]\ar@{-}[ld]\ar@{-}[d]&&\\
			\cdots&&3\ar@{--}[ll]&2&&3\ar@{--}[ll]&2&&3\ar@{--}[ll]&2&&\cdots\ar@{--}[ll]\\
			\ar@{-}[rrrrrrrrrrr]&&&&&&&&&&&\\
			&&&4\ar@{=}[uul]&&&4\ar@{=}[uul]&&&4\ar@{=}[uul]&&\\
			&&&1\ar@{-}[u]\ar@{-}[ld]\ar@{-}[d]&&&1\ar@{-}[u]\ar@{-}[ld]\ar@{-}[d]&&&1\ar@{-}[u]\ar@{-}[ld]\ar@{-}[d]&&\\
			\cdots&&3\ar@{--}[ll]&2&&3\ar@{--}[ll]&2&&3\ar@{--}[ll]&2&&\cdots\ar@{--}[ll]\\
			\ar@{-}[rrrrrrrrrrr]&&&&&&&&&&&\\
			&&&\cdots\ar@{=}[uul]&&&\cdots\ar@{=}[uul]&&&\cdots\ar@{=}[uul]&&
		}
		$\\
		\caption[]{Continued with Figure \ref{embed}. $\mathbb{T}^2$ is lifted to its universal covering space $\mathbb{R}^2$. Accordingly the embedded Markov chain is lifted to a covering space (not universal), which is a larger Markov chain.}
		\label{expand}
	\end{figure}

	The universal covering space of $\mathbb{T}^n$ is $\mathbb{R}^n$. Through the covering map, the embedded Markov chain is also lifted into $\mathbb{R}^n$ (Fig. \ref{expand}). The lifted Markov chain $L$ is a covering graph of $M$.  We can assign an $n$-tuple coordinate to each unit hypercube. In $L$, moving along special edges is the only way to change the coordinate.
	
	The construction is completed. Then we prove the existence of proper global potential in $L$ by showing that a trajectory is closed if and only if its potential gain is zero.
	
	Notice that a trajectory has zero potential gain if and only if the net number of each edge in $M$ is zero.
	
	Consider a closed trajectory in $L$. The net number of each special edge appears in the trajectory equals the net number of corresponding hypersurface crossed, which is zero. Fold this trajectory to $M$, and count the net number of each edge. If a non-special edge has non-zero net number, then on each end, there is a neighboring edge with non-zero net number. However, without special edges, non-special edges constitute the spanning tree, thus edges with non-zero net number cannot form a loop, a contradiction. 
	
	If a trajectory has zero potential gain, then first its two ends should be the same in $M$. Otherwise the total number of edges containing one end is odd, a contradiction. Also, if two ends are different in $L$, then at least one component of their coordinates is different, which means the net number of the corresponding special edge is not zero. Thus the potential gain is not zero.
	\qed
\end{proof}
\begin{proof}[Theorem \ref{T1}]
	A trajectory in $M$ has zero potential gain if and only if its lifting in $L$ has zero potential gain. As proved above, a trajectory in $L$ has zero potential gain if and only if it is closed. Furthermore, a closed trajectory in $\mathbb{T}^n$ is homotopy trivial if and only if its lifting in $\mathbb{R}^n$ is still closed. 
	\qed
\end{proof}

\begin{proof}[Proposition \ref{P1}]
	If $C$ is a covering graph of $L$ with covering map $p$, then $f$, a global potential on $L$, naturally induces a function $g$ on $C$ through $g(i)=f(p(i))$ with $i\in C$. Since covering map keeps neighborhood, $g$ is a global potential on $C$.
	
	Assume $C$ has a global potential $g$. We will construct a map $p$ from $C$ to $L$, and then prove it is surjective and locally bijective. 
	
	First choose $o'\in C$ and $o\in L$, such that they are mapped to the same state in $M$. We calibrate $g$ through adding a constant to $g$ such that $g(o')=f(o)$. For any state $i'\in C$, choose a trajectory connecting $i'$ and $o'$. Fold this trajectory to $M$, then lift it to $L$, such that the image of $o'$ in $M$ is lifted to $o$. Assume the image of $i'$ in $M$ is lifted to $i$. We have $g(i')=f(i)$. Since $f$ is a proper global potential, $i$ is the only state with $f(i)=g(i')$, thus $i$ is independent of the choice of trajectory. Hence we have a well-defined map $p$ from $C$ to $L$, which maps $i'$ to $i$. 
	
	For any state $j\in L$, choose a trajectory connecting $o$ and $j$. Fold this trajectory to $M$, then lift it to $C$, to get a trajectory connecting $o'$ and $j'$. Since $g(j')=f(j)$, $p(j')=j$, $p$ is surjective. 
	
	If for $j\in L$ and $j'\in C$, we have $g(j')=f(j)$ and thus $p(j')=j$, then for any neighbor $k$ of $j$, through folding-lifting of trajectory $o-\cdots-j-k$, we can find neighbor $k'$ of $j'$, such that $g(k')=f(k)$, thus $p(k')=k$. Similarly, for any neighbor $k'$ of $j'$, we can find neighbor $k$ of $j$ with $p(k')=k$. This shows that $p$ is locally bijective. Thus $p$ is a covering map.
	
	Uniqueness of lifting with proper global potential: Assume $C$ is another covering graph of $M$ with global potential, then $C$ is a covering graph of $L$, and a state $i\in L$ has at least two preimages in $C$. These preimages have the same potential, thus $C$ does not have proper global potential.
	\qed
\end{proof}

\section{Thermodynamic Quantities of Markov Chains}
\label{sec-3}

With a Markov chain and its lifting established,
we now consider the thermodynamic quantities, especially entropy productions of the 
corresponding Markov chains.

\subsection{Stationary distributions and measures of finite-state Markov chain and its lifting}

In the lifted chain $L$, state $i$ is lifted to $i_\alpha$, where $\alpha\in\mathbb{Z}^n$. The lifted initial distribution $p_{i_\alpha}(0)$ is compatible with the original distribution $\bar{p}_{ i }(0)$, in the sense that $\bar{p}_{ i }(0)=\sum_\alpha p_{i_\alpha}(0)$. Then we have $\bar{p}_{ i }(t)=\sum_\alpha p_{i_\alpha}(t)$.

$L$ has a periodic stationary measure $\pi_{i_\alpha}=\bar{\pi}_{ i }$, where $\bar{\pi}_{ i }$ is the unique stationary distribution of the original Markov chain $M$. Since $L$ has a global potential $\varphi_{i_\alpha}$, one could construct a detailed balanced stationary measure $\mu_{i_\alpha}=\exp(-\varphi_{i_\alpha})$, such that $\mu_{i_\alpha}q_{ij}=\mu_{j_\beta}q_{ji}$.

To further study the stationary distributions and measures, we need to consider the relative entropy of $p_i(t)$ with respect to any stationary measure $\theta_{i_\alpha}$,

$$\text{D}_{\text{KL}}(p,\theta)=\sum_{i}\sum_\alpha p_{i_\alpha}(t)\log\frac{p_{i_\alpha}(t)}{\theta_{i_\alpha}}.$$

{We stipulate that $0\log 0=0$.}

In this paper, the initial distribution $p_{i_\alpha}(0)$ is concentrated enough such that {$\text{D}_{\text{KL}}(p(0),\theta)$ is finite for all $\theta$. This guarantees that all quantities below are finite in finite time.} 

The following Lemma \ref{voigt} can be found in many references \cite{morimoto,voigt,thompson-qian}. 

{The notation $j_\beta\sim i_\alpha$ means that $j_\beta$ and $i_\alpha$ are adjacent. }

\begin{lemma}
	\label{voigt}
	$\text{D}_{\text{KL}}(p,\theta)$ is monotonically decreasing.
\end{lemma}
\begin{proof}
	\begin{eqnarray*}
		&& \frac{\mathrm{d}}{\mathrm{d}t}\sum_{i}\sum_\alpha p_{i_\alpha}(t)\log\frac{p_{i_\alpha}(t)}{\theta_{i_\alpha}} 
		\\
		&=& \sum_{i}\sum_\alpha \frac{\mathrm{d}p_{i_\alpha}(t)}{\mathrm{d}t}\log\frac{p_{i_\alpha}(t)}{\theta_{i_\alpha}} +\sum_{i}\sum_\alpha \frac{\mathrm{d}p_{i_\alpha}(t)}{\mathrm{d}t}
		\\
		&=& -\sum_{i}\sum_\alpha \sum_{j_\beta\sim i_\alpha} [p_{i_\alpha}(t)q_{ij}-p_{j_\beta}(t)q_{ji}]\log\frac{p_{i_\alpha}(t)}{\theta_{i_\alpha}}
		\\
		&=&-\frac{1}{2}\sum_{i}\sum_\alpha \sum_{j_\beta\sim i_\alpha} \left[p_{i_\alpha}(t)q_{ij}-p_{j_\beta}(t)q_{ji}\right]\left[\log\frac{p_{i_\alpha}(t)}{\theta_{i_\alpha}}-\log\frac{p_{j_\beta}(t)}{\theta_{j_\beta}}\right]
	\end{eqnarray*}
	\begin{eqnarray*}
		&=& -\sum_{i}\sum_\alpha \sum_{j_\beta\sim i_\alpha} p_{i_\alpha}(t)q_{ij}\log\frac{p_{i_\alpha}(t) \theta_{j_\beta}}{\theta_{i_\alpha} p_{j_\beta}(t)}\le -\sum_{i}\sum_\alpha \sum_{j_\beta\sim i_\alpha} p_{i_\alpha}(t)q_{ij}\left[1-\frac{\theta_{i_\alpha} p_{j_\beta}(t)}{p_{i_\alpha}(t) \theta_{j_\beta}}\right]
		\\
		&=& -\sum_i\sum_j\bar{p}_i(t) q_{ij}+\sum_{j}\sum_\beta \sum_{i_\alpha\sim j_\beta} \frac{q_{ij}\theta_{i_\alpha} p_{j_\beta}(t)}{\theta_{j_\beta}}
		\\
		&=&-\sum_i\sum_j\bar{p}_i(t) q_{ij}+\sum_j \sum_\beta \frac{p_{j_\beta}(t)}{\theta_{j_\beta}}\sum_i q_{ji}\theta_{j_\beta} 
		\\
		&=& -\sum_i\sum_j\bar{p}_i(t) q_{ij}+\sum_j\sum_i\bar{p}_j(t) q_{ji}=0.
	\end{eqnarray*}

	The inequality is from $\log y\ge 1-1/y$. The equality holds if and only if $p_{i_\alpha}(t)=c\theta_{i_\alpha}$ for a constant $c$.
	\qed
\end{proof}

Then we can prove the following result:
\begin{proposition}
	\label{nsm}
	$L$ has no stationary probability distribution.
\end{proposition}
\begin{proof}
	Assume there is a stationary probability distribution $\eta_{i_\alpha}$. Let $p(t)$ be the stationary distribution, and $\theta$ be $\pi$, then $\text{D}_{\text{KL}}(p,\pi)$ is a constant. This is true only if the equality holds in Lemma \ref{voigt}, which means $\eta_{i_\alpha}\pi_j=\eta_{j_\beta}\pi_i$. Thus $\eta$ and $\pi$ only differ by a constant multiple. $\pi$ is non-normalizable, so is $\eta$.
	\qed
\end{proof}

\subsection{Instantaneous entropy production rate, free energy, and housekeeping heat}

For $L$ with probability distribution $p_{i_\alpha}(t)$, one could define several thermodynamic quantities: entropy production rate, free energy, and housekeeping heat. 

\begin{definition}
	The instantaneous free energy with respect to stationary measure $\theta$, $F^\theta(t)$, is defined as $F^\theta(t)=\text{D}_{\text{KL}}(p(t),\theta)$. Its time derivative is
	$$\mathrm{d}F^\theta(t)/\mathrm{d}t=-\frac{1}{2}\sum_{j_\beta}\sum_{i_\alpha\sim j_\beta} [p_{i_\alpha}(t)q_{ij}-p_{j_\beta}(t)q_{ji}]\log\frac{p_{i_\alpha}(t) \theta_{j_\beta}}{p_{j_\beta}(t) \theta_{i_\alpha}}.$$
\end{definition}
\begin{definition}
	The instantaneous entropy production rate $e_p(t)$ is defined as \cite{JQQ}
	$$e_p(t)=\frac{1}{2}\sum_{j_\beta}\sum_{i_\alpha\sim j_\beta}\big[p_{i_\alpha}(t)q_{ij}-p_{j_\beta}(t)q_{ji}\big]\log\frac{p_{i_\alpha}(t)q_{ij}}{p_{j_\beta}(t)q_{ji}}.$$
\end{definition}
This definition is derived from the original idea of entropy production rate that it describes the difference between a process and its time inverse.
\begin{definition}
	The instantaneous housekeeping heat with respect to stationary measure $\theta$, $Q_{hk}^\theta(t)$, is defined as 
	$$Q_{hk}^\theta(t)=e_p(t)+\mathrm{d}F^\theta(t)/\mathrm{d}t=\frac{1}{2}\sum_{j_\beta}\sum_{i_\alpha\sim j_\beta}\big[p_{i_\alpha}(t)q_{ij}-p_{j_\beta}(t)q_{ji}\big]\log\frac{\theta_{i_\alpha}q_{ij}}{\theta_{j_\beta}q_{ji}}.$$
\end{definition}
In general stationary measure is not unique, therefore one could have different versions of free energy and housekeeping heat.

We can also define these thermodynamic quantities for $M$ with $\theta=\bar{\pi}$:  $\bar{F}(t)$, $\bar{e_p}(t)$ and $\bar{Q}_{hk}(t)$.

From Lemma \ref{voigt}, $\mathrm{d}F^\theta(t)/\mathrm{d}t\le 0$. From the definition of instantaneous entropy production rate, $e_p(t)\ge 0$. For $Q_{hk}^\theta(t)$, we have the same result.

\begin{proposition}
	$Q_{hk}^\theta(t)\ge 0$.
\end{proposition}
\begin{proof}
	\begin{eqnarray*}
		Q_{hk}^\theta(t) &=& \sum_{i}\sum_\alpha \sum_{j_\beta\sim i_\alpha} p_{i_\alpha}(t)q_{ij}\log\frac{\theta_{i_\alpha} q_{ij}}{\theta_{j_\beta} q_{ji}}\ge \sum_{i}\sum_\alpha \sum_{j_\beta\sim i_\alpha} p_{i_\alpha}(t)q_{ij}\left[1-\frac{\theta_{j_\beta} q_{ji}}{\theta_{i_\alpha} q_{ij}}\right]
		\\
		&=& \sum_i\sum_j\bar{p}_i(t) q_{ij}-\sum_i \sum_\alpha \frac{ p_{i_\alpha}(t)}{\theta_{i_\alpha}}\sum_{j_\beta\sim i_\alpha}q_{ji}\theta_{j_\beta} 
		\\
		&=&\sum_i\sum_j\bar{p}_i(t) q_{ij}-\sum_i \sum_\alpha \frac{ p_{i_\alpha}(t)}{\theta_{i_\alpha}}\sum_{j}q_{ij}\theta_{i_\alpha}
		\\
		&=& \sum_i\sum_j\bar{p}_i(t) q_{ij}-\sum_i\sum_j\bar{p}_i(t) q_{ij}=0.
	\end{eqnarray*}
	
	The inequality is from $\log y\ge 1-1/y$.
	\qed
\end{proof}

Thus we have the decomposition
$$e_p(t)=Q_{hk}^\theta(t)+[-\mathrm{d}F^\theta(t)/\mathrm{d}t],$$
where each term is non-negative.

\subsection{Time limits of thermodynamic quantities}

Since $\bar{p}(t)$ converges to $\bar{\pi}$, $\bar{F}(t)$ and $\mathrm{d}\bar{F}(t)/\mathrm{d}t$ converge to $0$, $\bar{e_p}(t)$ and $\bar{Q}_{hk}(t)$ converge to the stationary entropy production rate 
$$\bar{e_p}=\frac{1}{2}\sum_i\sum_j[\bar{\pi}_iq_{ij}-\bar{\pi}_j q_{ji}]\log\frac{\bar{\pi}_iq_{ij}}{\bar{\pi}_j q_{ji}}.$$

For $L$, $p(t)$ does not converge to a stationary distribution, therefore we do not have the stationary version of these quantities. However, we can still study their behavior as $t\to \infty$.

If we set $\theta$ to be the periodic stationary measure $\pi$, then $Q_{hk}^\pi(t)$ converges to
$$\frac{1}{2}\sum_i\sum_j[\bar{\pi}_iq_{ij}-\bar{\pi}_j q_{ji}]\log\frac{\bar{\pi}_iq_{ij}}{\bar{\pi}_j q_{ji}},$$
which is just $\bar{e_p}$.

If we set $\theta$ to be the detailed balanced stationary measure $\mu$, then $Q_{hk}^\mu(t)\equiv 0$ since $\mu_{i_\alpha}q_{ij}=\mu_{j_\beta}q_{ji}$.

For the time limit of $e_p(t)$, we have the following theorem. The proof is in the next part.

For a probability distribution $p_i$, its entropy is 
$$\text{H}[p]=\sum_{i} -p_{i}\log p_{i}.$$

\begin{theorem}
	\label{t2}
	Assume the initial distribution $p_{i_\alpha}(0)$ has finite entropy. Then $e_p(t)$ converges to $\bar{e_p}$ in Ces\`{a}ro's sense that $\lim_{T\to\infty}\frac{1}{T}\int_0^T e_p(t)\mathrm{d}t=\bar{e_p}$.
\end{theorem}

In summary, $e_p=Q^\theta_{hk}+(-\mathrm{d}F^\theta/\mathrm{d}t)$, where $e_p$, $Q^\theta_{hk}$ and $-\mathrm{d}F^\theta/\mathrm{d}t$ are non-negative.

$e_p\to\bar{e_p}$ in Ces\`{a}ro's sense, $\mathrm{d}F^\mu/\mathrm{d}t=-e_p\to -\bar{e_p}$ in Ces\`{a}ro's sense, $Q_{hk}^\mu\equiv 0$, $\mathrm{d}F^\pi/\mathrm{d}t\to 0$ in Ces\`{a}ro's sense, $Q_{hk}^\pi\to \bar{e_p}$ in general sense.

Therefore, the periodic stationary measure $\pi$ and the detailed balanced stationary measure $\mu$ reach the maximum and minimum of $Q_{hk}^\theta$, namely $\bar{e_p}$ and $0$, as $t\to \infty$.

\subsection{Proof of Theorem \ref{t2}}
The proof consists of the following Lemmas \ref{le8}-\ref{l00}.

The key idea is to transform the convergence of entropy production rate into the control of entropy.

\begin{lemma}
	Assume the initial distribution $p(0)$ has finite entropy. Then 
	$$\lim_{T\to\infty}\frac{1}{T}\int_0^T e_p(t)\mathrm{d}t=\bar{e_p}\Longleftrightarrow \lim_{T\to\infty}\text{H}[p(T)]/T=0.$$
	\label{le8}
\end{lemma}
\begin{proof}
	We have 
	$$\frac{\mathrm{d}F^\pi(t)}{\mathrm{d}t}+e_p(t)=\frac{1}{2}\sum_i\sum_j[\bar{p}_i(t)q_{ij}-\bar{p}_j(t)q_{ji}]\log\frac{\bar{\pi}_i q_{ij}}{\bar{\pi}_j q_{ji}}=\frac{\mathrm{d}\bar{F}(t)}{\mathrm{d}t}+\bar{e_p}(t).$$
	
	Therefore
	\begin{eqnarray*}
		&& \frac{1}{T}\int_0^T e_p(t)\mathrm{d}t-\frac{1}{T}\int_0^T \bar{e_p}(t)\mathrm{d}t
		\\
		&=& \frac{1}{T}\Big[F^\pi(0)-F^\pi(T)+\bar{F}(T)-\bar{F}(0)\Big].
	\end{eqnarray*}

	Since $\bar{p}(t)$ converges to $\bar{\pi}$ and $\bar{F}(t)$ converges to $0$, $\bar{F}(T)$ is bounded.
	
	$F^\pi(0)-\bar{F}(0)=\text{H}[\bar{p}(0)]-\text{H}[p(0)]$, which is finite. 
	
	$F^\pi(T)+\text{H}[p(T)]=-\sum_i \bar{p}_i(T) \log \bar{\pi}_i$, which is bounded. Thus 
	$$\lim_{T\to\infty}F^\pi(T)/T=-\lim_{T\to\infty}\text{H}[p(T)]/T.$$
	
	We also have $\lim_{T\to\infty}\frac{1}{T}\int_0^T \bar{e_p}(t)\mathrm{d}t=\bar{e_p}$.
	
	Therefore 
	$$\lim_{T\to\infty}\frac{1}{T}\int_0^T e_p(t)\mathrm{d}t-\bar{e_p}=-\lim_{T\to\infty}\text{H}[p(T)]/T.$$
	
	Since the initial distribution has finite entropy, all quantities used are finite.
	
	\qed
\end{proof}

From this proof, we can see that $e_p(t)\to \bar{e_p}\Longleftrightarrow \mathrm{d}\text{H}[p(t)]/\mathrm{d}t\to 0$. The techniques we use can prove $\text{H}[p(t)]/t\to 0$, but not $\mathrm{d}\text{H}[p(t)]/\mathrm{d}t\to 0$. Thus we do not have $e_p(t)\to \bar{e_p}$.

\begin{lemma}
	There exists a constant $C$, such that if $L$ starts from a single state, $\text{H}[p(t)]\le C\log t$.
	\label{llo}
\end{lemma}

\begin{proof}
	Without loss of generality, we assume that $\max\{q_i\}\le 1$. The idea is to show that at time $t$, the contribution to entropy from states close to initial state (no more than $3t$ steps away) is controlled by $C\log t$, and the contribution to entropy from states far away from initial state is less than $1$. We will use an easy lemma: if at most $ m \ge 3$ states have total probability at most $p$, then the largest possible entropy from these states is $-p\log(p/ m )$, which corresponds to uniform distribution.
	
	Since $\max\{q_i\}\le 1$, at time $t$, the total number of jumping $J_t$ is controlled by a Poisson variable with parameter $t$, $Y_t$. This means that for any $l>0$, $\mathbb{P}(J_t>l)\le\mathbb{P}(Y_t>l)$.
	
	Fix $t$. For any $ m \ge 3t$, 
	$$\mathbb{P}(J_t\ge m )\le \sum_{k= m }^\infty \frac{t^k}{k!}e^{-t}\le e^{-t}\sum_{k= m }^\infty \frac{t^ m }{ m !}\frac{1}{2^{k- m }}=2e^{-t} \frac{t^ m }{ m !}.$$
	
	From Striling's formula on lower bound of factorial, $ m !\ge\sqrt{2\pi m } ( m /e)^ m $.
	
	Thus 
	$$\mathbb{P}(J_t\ge m )\le e^{-t}\Big(\frac{et}{ m }\Big)^ m \le \Big(\frac{e}{3}\Big)^ m .$$
	
	If the lifted chain at time $t$, $L_t$, is $ m $ steps away from the initial state, then it must jump no less than $ m $ steps. Hence 
	$$\mathbb{P}(\text{d}(L_t,o)= m )\le\mathbb{P}(J_t\ge m )\le\Big(\frac{e}{3}\Big)^ m .$$
	
	Here $\text{d}$ is the step counting distance, and $o$ is the initial state.
	
	For any lifting, the number of states satisfying $\text{d}(L_t,o)=m$ is less by $k^ m $, where $k$ is the number of states in $M$. Thus the entropy from these states is no larger than
	$$-\Big(\frac{e}{3}\Big)^ m \log \Big(\frac{e}{3k}\Big)^ m .$$
	
	The entropy from states with $\text{d}(L_t,o)\ge 3t$ is controlled by
	$$\sum_{ m =\lceil 3t\rceil}^\infty -\Big(\frac{e}{3}\Big)^ m \log \Big(\frac{e}{3k}\Big)^ m \le \log \frac{3k}{e}\Big[\frac{\lceil 3t\rceil}{1-\frac{e}{3}}\Big(\frac{e}{3}\Big)^{\lceil 3t\rceil}+\frac{1}{(1-\frac{e}{3})^2}\Big(\frac{e}{3}\Big)^{\lceil 3t\rceil+1}\Big]$$
	$$\le \log(2k)\Big[(400+20\lceil 3t\rceil)\Big(\frac{e}{3}\Big)^{\lceil 3t\rceil}\Big].$$
	
	When $t$ is large enough, the above term is less than $1$.
	
	There are at most $k(6t+1)^n$ states satisfying $\text{d}(L_t,o)\le 3t$, where $n$ is the first Betti number of $M$. The entropy from such states is at most $\log[k(6t+1)^n]\le C\log t$ for some large constant $C$. 
	\qed
\end{proof}

\begin{remark}
	The fact that $\text{H}[p(t)]/t\to 0$ is a corollary of conclusions in a paper by Kaimanovich and Woess \cite{Kai}. We still prove it here since it provides a concrete bound of the entropy, and an estimation of convergence rate. In the Kaimanovich-Woess paper \cite{Kai}, the lifted Markov chains are named ``quasi-homogeneous chains''. If $M$ is reduced to a single state with several cycles, then $L$ becomes a random walk on a group. For random walks on groups, whether the entropy grows linearly is a well-studied problem \cite{erschler2003drift,saloff2016random}. For a Markov chain with at least two cycles, its universal cover has $\lim \text{H}[p(t)]/t>0$ \cite{Kai}. Nevertheless, how to express this limit explicitly from $\{q_{ij}\}$ is still open.
\end{remark}

The following lemma is directly from Jensen's inequality and the fact that $-p\log p$ is convex up. It is used to extend the result of Lemma \ref{llo} to general initial distributions.
\begin{lemma}
	\label{l00}
	Consider a sequence of probabilities $p_i$ with sum $1$, and a sequence of discrete distributions $\mu_i$ on the same space. Then 
	$$\text{H}(\sum p_i \mu_i)\le -\sum p_i \log p_i +\sum p_ih(\mu_i).$$
\end{lemma}

If the initial distribution $p(0)$ has entropy $h_0$, then the entropy at time $t$ is less than $h_0+C\log t$. Hence $\text{H}[p(t)]/t\to 0$.

This finishes the proof of Theorem \ref{t2}. In fact, this result is true as long as the lifted chain grows sub-exponentially.

\subsection{Entropy production as energy dissipation}
Consider the free energy with detailed balanced stationary measure $\mu=\exp(-\varphi)$
$$F^\mu(t)=\sum_{i_\alpha}p_{i_\alpha}(t)\log\frac{p_{i_\alpha}(t)}{\mu_{i_\alpha}}=\sum_{i_\alpha}p_{i_\alpha}(t)\log p_{i_\alpha}(t)+\sum_{i_\alpha}p_{i_\alpha}(t)\varphi_{i_\alpha}.$$

Here $\text{H}(t)=-\sum_{i_\alpha}p_{i_\alpha}(t)\log p_{i_\alpha}(t)$ is the entropy, and $E(t)=\sum_{i_\alpha}p_{i_\alpha}(t)\varphi_{i_\alpha}$ is the mean potential energy. Thus $F^\mu(t)=E(t)-\text{H}(t)$.

For $E(t)$, we have the following result:
\begin{proposition}
	The time derivative of $E(t)$ converges to the negative stationary entropy production rate,
	$$\frac{\mathrm{d}E(t)}{\mathrm{d}t}\to -\bar{e_p}.$$
\end{proposition}
\begin{proof}
	$$\frac{\mathrm{d}E(t)}{\mathrm{d}t}+\bar{e_p}(t)=\frac{1}{2}\sum_{j_\beta}\sum_{i_\alpha\sim j_\beta}[p_{i_\alpha}(t)q_{ij}-p_{j_\beta}(t)q_{ji}]\log\frac{\mu_{i_\alpha}\bar{p}_i(t)q_{ij}}{\mu_{j_\beta}\bar{p}_j(t)q_{ji}}$$
	$$\to \frac{1}{2}\sum_i\sum_j(\bar{\pi}_iq_{ij}-\bar{\pi}_jq_{ji})\log\frac{\bar{\pi}_i}{\bar{\pi}_j}.$$
	The last term is the time derivative of $\sum_i \bar{p}_i(t)\log\bar{p}_i(t)$ when $\bar{p}_i(t)=\bar{\pi}_i$, which is $0$.
	\qed
\end{proof}

In the decomposition of free energy $F^\mu(t)=E(t)-\text{H}(t)$, The first term is asymptotically linear with $t$, and the second term is sub-linear with $t$ (controlled by $C\log t$).

The entropy production of the finite Markov chain, which cannot be described by system status quantities directly, is reflected by the free energy/potential energy dissipation of the lifted Markov chain.

\subsection{Mutual information}
For two random variables$X,Y$, we can define their mutual information as 
$$\text{MI}(X,Y)=\text{H}(X)+\text{H}(Y)-\text{H}(\{X,Y\}).$$
Mutual information is commonly used in information theory, and describes the information shared by $X$ and $Y$. It can be also interpreted as the prediction power of $X$ on $Y$ (and $Y$ on $X$ vice versa). \text{MI}(X,Y) is nonnegative, and equals $0$ if and only if $X$ and $Y$ are independent.

If $L_t$ starts from a preimage of $i$, then we denote its entropy at time $t$ as $\text{H}_t^i$, and the mutual information between $L_1$ and $L_t$ by $\text{MI}^i(L_1,L_t)$. $\text{MI}^i(L_1,L_t)$ measures the impact of $L_1$ on far future. We have the following result:
\begin{proposition}
	\label{P5}
	For the lifted Markov chain $L_t$, $\lim_{t\to \infty}\text{MI}^i(L_1,L_t)=0$.
\end{proposition}
This means that $L_t$ is asymptotically independent with any past. Also, $L_1$ has no information about the behavior at infinity.

Proposition \ref{P5} follows directly from Lemma \ref{llo} and Theorem \ref{Tn} below.

Theorem \ref{Tn} is valid not only for $L$, but also for any transitive lifting of $M$, denoted as $X_t$. Here transitive means that for any state $i$ in $M$, and any two preimages $i_1,i_2$ of $i$ in the lifted chain, there is an automorphism of the lifted chain, such that $i_1$ is mapped to $i_2$. Transitivity guarantees that $\text{H}_t^{i_1}=\text{H}_t^{i_2}$ and $\text{MI}^{i_1}(X_1,X_t)=\text{MI}^{i_2}(X_1,X_t)$.

Similar (and much simpler) results are also valid for random walks on groups \cite{kaimanovich1983random}.
\begin{theorem}
	Consider any transitive lifting of $M$. For any $i,j$, $\lim_{t\to \infty}\text{H}_t^i/t$ and $\lim_{t\to \infty}\text{MI}^j(X_1,X_t)$ exist. Also, 
	$$\lim_{t\to \infty}\text{H}_t^i/t>0 \Longleftrightarrow \lim_{t\to \infty}\text{MI}^j(X_1,X_t)>0.$$
	\label{Tn}
\end{theorem}

\begin{proof}
	
	In the following, the summation without superscript $\sum_i$ takes all vertices in the lifted Markov chain. 
	
	Due to the Markov property, for $1<t<s$, $X_1$ and $X_s$ are independent conditioned on $X_t$. From data processing inequality \cite{DPI}, $\text{MI}^j(X_1,X_t)$ is decreasing with $t$ when $t>1$. Since $\text{MI}^j(X_1,X_t)\ge 0$, $\lim_{t\to \infty}\text{MI}^j(X_1,X_t)$ exists.
	
	The existence of $\lim_{t\to \infty}\text{H}_t^i/t$ and the equivalence of $\lim_{t\to \infty}\text{H}_t^i/t>0$ and $\lim_{t\to \infty}\text{MI}^j(X_1,X_t)>0$ follow from the following lemmas.
	
	\begin{lemma}
		The averaged entropy and mutual information have the following relationship: $$\lim_{t\to \infty}\sum_{i=1}^k \bar{\pi}_i\text{H}_t^i/t=\lim_{t\to \infty}\sum_{i=1}^k \bar{\pi}_i\text{MI}^i(X_1,X_t).$$
		\label{l1}
	\end{lemma}
	
	\begin{lemma}
		If $\lim_{t\to \infty}\sum_{i=1}^k \bar{\pi}_i\text{H}_t^i/t=C_1$, then for each $i$, $\lim_{t\to \infty}\text{H}_t^i/t$ exists, and equals $C_1$.
		\label{l2}
	\end{lemma}
	
	\begin{lemma}
		If $\lim_{t\to \infty}\text{MI}^a(X_1,X_t)>0$ for one $a$, then $\lim_{t\to \infty}\text{MI}^b(X_1,X_t)>0$ for any $b$.  
		\label{l3}
	\end{lemma}
	
	From Lemma \ref{l1} and Lemma \ref{l2}, if $\lim_{t\to \infty}\text{MI}^j(X_1,X_t)>0$, then $\lim_{t\to \infty}\sum_{i=1}^k \bar{\pi}_i\text{H}_t^i/t>0$, thus $\lim_{t\to \infty}\text{H}_t^i/t>0$. 
	
	If $\lim_{t\to \infty}\text{H}_t^i/t>0$, then $\lim_{t\to \infty}\sum_{i=1}^k \bar{\pi}_i\text{MI}^i(X_1,X_t)>0$. From Lemma \ref{l3} we have 
	$$\lim_{t\to \infty}\text{MI}^j(X_1,X_t)>0.$$
	\qed
\end{proof}

\begin{proof}[Lemma \ref{l1}]
	We have
	$$\text{MI}^j(X_1,X_t)$$
	$$=\sum_i\sum_l\mathbb{P}(X_1=i,X_t=l\mid X_0=j)\log\frac{\mathbb{P}(X_1=i,X_t=l\mid X_0=j)}{\mathbb{P}(X_1=i\mid X_0=j)\mathbb{P}(X_t=l\mid X_0=j)}$$
	$$=-\sum_l\mathbb{P}(X_t=l\mid X_0=j)\log\mathbb{P}(X_t=l\mid X_0=j)$$
	$$+\sum_i\sum_l\mathbb{P}(X_1=i\mid X_0=j)\mathbb{P}(X_t=l\mid X_0=j, X_1=i)\log\mathbb{P}(X_t=l\mid X_0=j, X_1=i)$$
	$$=\text{H}^j_t+\sum_i\mathbb{P}(X_1=i\mid X_0=j)\sum_l\mathbb{P}(X_t=l\mid X_1=i)\log\mathbb{P}(X_t=l\mid  X_1=i)$$
	$$=\text{H}^j_t-\sum_i\mathbb{P}(X_1=i\mid X_0=j)\text{H}_{t-1}^i=\text{H}^j_t-\sum_{i=1}^k\mathbb{P}(\bar{X}_1=i\mid \bar{X}_0=j)\text{H}_{t-1}^i.$$
	Notice that for Markov chain $\bar{X}_t$, we have $\sum_{j=1}^{k}\bar{\pi}_j\mathbb{P}(X_1=i\mid X_0=j)=\bar{\pi}_i$. Thus
	$$\sum_{j=1}^{k}\bar{\pi}_j\text{MI}^j(X_1,X_t)=\sum_{j=1}^{k}\bar{\pi}_j\text{H}^j_t-\sum_{j=1}^{k}\bar{\pi}_j\sum_{i=1}^k\mathbb{P}(\bar{X}_1=i\mid \bar{X}_0=j)\text{H}_{t-1}^i$$
	$$=\sum_{j=1}^{k}\bar{\pi}_j\text{H}^j_t-\sum_{i=1}^{k}\bar{\pi}_i\text{H}^i_{t-1}=\sum_{j=1}^{k}\bar{\pi}_j(\text{H}^j_t-\text{H}^j_{t-1}).$$
	
	The limit $\lim_{t\to \infty}\sum_{i=1}^k \bar{\pi}_i\text{MI}^i(X_1,X_t)$ exists, then
	$$\lim_{t\to \infty}\sum_{j=1}^{k}\bar{\pi}_j(\text{H}^j_t-\text{H}^j_{t-1})=\lim_{t\to \infty}\sum_{i=1}^k \bar{\pi}_i\text{MI}^i(X_1,X_t).$$
	Therefore
	$$\lim_{t\to \infty}\sum_{i=1}^k \bar{\pi}_i\text{H}_t^i/t=\lim_{t\to \infty}\sum_{i=1}^k \bar{\pi}_i\text{MI}^i(X_1,X_t).$$
	\qed
\end{proof}

\begin{proof}[Lemma \ref{l2}]
	Since $\lim_{t\to \infty}\sum_{i=1}^k \bar{\pi}_i\text{H}_t^i/t$ exists, and $\text{H}_t^i$ will not blow up in finite time, $\sum_{i=1}^k \bar{\pi}_i\text{H}_t^i/t$ is bounded. Since $\text{H}_t^i/t\ge 0$, we have $\text{H}_t^i/t<C_2$ with a constant $C_2$ for any $i$. 
	
	Fix any $\epsilon>0$. The original Markov chain $\bar{X}_t$ is irreducible, thus starting from state $i$, $\mu_t^i$, the distribution at time $t$, will converge to $\bar{\pi}$. Choose $\Delta t$ large enough such that the total variation distance $\text{d}_{TV}(\mu_{\Delta t}^i,\bar{\pi})<\epsilon/(2C_2)$.
	
	For a group of probability distributions $\mu_0,\mu_1,\ldots,\mu_k$ and probabilities $p_1,\ldots,p_k$ with $\sum_{i=1}^{k}=1$ and $\mu_0=\sum_{i=1}^k p_i\mu_i$, we have 
	$$\text{H}(\mu_0)\le \sum_{i=1}^{k}p_i\text{H}(\mu_i)-\sum_{i=1}^{k}p_i\log p_i.$$
	
	$$\text{H}_t^i\le \sum_{j=1}^k \mathbb{P}(\bar{X}_{\Delta t}=j\mid \bar{X}_0=i)\text{H}_{t-\Delta t}^j$$
	$$-\sum_{j=1}^k \mathbb{P}(\bar{X}_{\Delta t}=j\mid \bar{X}_0=i)\log \mathbb{P}(\bar{X}_{\Delta t}=j\mid \bar{X}_0=i)$$
	$$\le \sum_{j=1}^k \bar{\pi}_j \text{H}_{t-\Delta t}^j+(t-\Delta t)\epsilon/2+\log k.$$
	
	Set $T$ large enough such that $\sum_{j=1}^k \bar{\pi}_j \text{H}_{T-\Delta t}^j/(T-\Delta t)<C_1+\epsilon/4$, and $(\log k)/T<\epsilon/4$. Now
	$\text{H}_T^i/T<C_1+\epsilon$.
	
	Choose $T'$ large enough such that $\text{H}_{T'}^j/T'<C_1+\epsilon \bar{\pi}_i/2$ for any $j$, and $\sum_{j=1}^k \bar{\pi}_j \text{H}_{T'}^j/T'>C_1-\epsilon\bar{\pi}_i/2$. Then $\text{H}_{T'}^i/T'>C_1-\epsilon$.
	
	Thus $\lim_{t\to\infty}\text{H}_t^i/t$ exists, and equals $C_1$.
	\qed
\end{proof}

\begin{proof}[Lemma \ref{l3}]
	Assume $\lim_{t\to \infty}\text{MI}^a(X_1,X_t)=C_3$.
	
	Since $\text{MI}^a(X_1,X_t)$ is decreasing with $t>1$, $\text{MI}^a(X_1,X_t)\ge C_3$.

	Assume that starting from $a$, $X_1$ is at state $s_j$ with probability $p_j$. Reorder the states such that $\{p_j\}$ is decreasing. Since $\text{H}_1^i$ is finite, we could choose $n$ large enough, such that 
	$$-\sum_{l>n}p_l\log p_l<C_3/3\;\;\text{and}\;\;-\Big(\sum_{l\le n}p_l\Big)\log \Big(\sum_{l\le n}p_l\Big)<C_3/3.$$
	
	Set $\mathbb{P}(X_t=s_j\mid X_1=s_i)=y_{ij}$. For any $t>1$, construct $\ddot{X}_1$ and $\ddot{X}_t$ with joint distribution $\mathbb{P}(\ddot{X}_1=s_i,\ddot{X}_t=s_j)=y_{ij}p_i/\sum_{l\le n}p_l$ for $i\le n$ and any $j$.
	\begin{lemma}
		For any $t>1$, $\MI(\ddot{X}_1,\ddot{X}_t)\ge C_3/3.$
		\label{l4}
	\end{lemma}
	Now consider the process $X_t$ starting from state $b$. Assume $\mathbb{P}(X_1=s_i)=p_i'$. Then the joint probability  is $\mathbb{P}(X_1=s_i,X_t=s_j)=p_i'y_{ij}$.
	
	For any $t>1$, construct $\tilde{X}_1$ and $\tilde{X}_t$ with joint distribution $\mathbb{P}(\tilde{X}_1=s_i,\tilde{X}_t=s_j)=y_{ij}p_i'/\sum_{l\le n}p_l'$ for $i\le n$ and any $j$.
	
	\begin{lemma}
		There exists $C_4>0$ such that for any $t>1$, $\MI(\tilde{X}_1,\tilde{X}_t)\ge \MI(\ddot{X}_1,\ddot{X}_t)/C_4$.
		\label{l5}
	\end{lemma}
	\begin{lemma}
		There exists $C_5>0$ such that for any $t>1$, $\MI^b(X_1,X_t)\ge C_5\MI(\tilde{X}_1,\tilde{X}_t)$.
		\label{l6}
	\end{lemma}
	
	Combining Lemmas \ref{l4}, \ref{l5}, \ref{l6}, then for any $t>1$, 
	$$\MI^b(X_1,X_t)\ge \frac{C_3C_5}{3C_4}>0.$$

	Thus
	$$\lim_{t\to\infty}\MI^b(X_1,X_t)\ge \frac{C_3C_5}{3C_4}>0.$$
	\qed
\end{proof}

\begin{proof}[Lemma \ref{l4}]
	Consider $X_t$ starting from $a$. For $t>1$, from $X_t$, we build $X^*_t$ as following: $X^*_t=X_t$ if $X_1=s_1,\ldots,s_n$, and $X^*_t=s^*_m$ if $X_1=s_m$ for $m>n$. Here $s^*_m$ is a virtue state, different from any $s_i$.
	
	First we prove
	$$\text{MI}(X_1,X^*_t)=\MI(\ddot{X}_1,\ddot{X}_t)(\sum_{i\le n}p_i)-\Big(\sum_{i\le n}p_i\Big)\log\Big(\sum_{i\le n}p_i\Big)-\sum_{i>n}p_i\log p_i.$$
	\begin{eqnarray*}
		&& \text{MI}(X_1,X^*_t)-\MI(\ddot{X}_1,\ddot{X}_t)\Big(\sum_{i\le n}p_i\Big)
		= \sum_j\sum_{i\le n}p_i y_{ij} \log \frac{y_{ij}}{\sum_{l\le n}p_ly_{lj}}
		\\
		&& -\sum_{i>n}p_i\log p_i-\Big(\sum_{m\le n}p_m\Big)\sum_j\sum_{i\le n}\frac{p_i}{\sum_{l\le n}p_l}y_{ij}\log\frac{y_{ij}\sum_{l\le n}p_l}{\sum_{l\le n}p_ly_{lj}}
		\\
		&=& -\sum_j\sum_{i\le n}p_i y_{ij} \log \sum_{l\le n}p_l-\sum_{i>n}p_i\log p_i
		\\
		&=& -\Big(\sum_{i\le n}p_i\Big)\log\Big(\sum_{i\le n}p_i\Big)-\sum_{i>n}p_i\log p_i.
	\end{eqnarray*}
	
	Then we prove 
	$$\text{MI}(X_1,X_t)\le \text{MI}(X_1,X^*_t).$$

	$$\text{MI}(X_1,X^*_t)-\text{MI}(X_1,X_t)$$
	
	$$=\sum_{i\le n}\sum_j p_jy_{ij}\log\frac{p_iy_{ij}}{p_i\sum_{k\le n}(p_ky_{kj})}-\sum_{i>n}p_i\log p_i$$
	$$-\sum_{i\le n}\sum_j p_jy_{ij}\log\frac{p_iy_{ij}}{p_i\sum_k(p_ky_{kj})}-\sum_{i> n}\sum_j p_jy_{ij}\log\frac{p_iy_{ij}}{p_i\sum_k(p_ky_{kj})}$$
	$$=\sum_j\sum_{i\le n}p_jy_{ij} \log\frac{\sum_k p_k y_{kj}}{\sum_{k\le n} p_ky_{kj}}-\sum_{i>n}p_i\log p_i-\sum_j \sum_{i>n} p_iy_{ij}\log (p_iy_{ij})$$
	$$+\sum_{i>n}\sum_j p_i y_{ij}\log p_i+\sum_j\Big(\sum_{i>n}p_iy_{ij}\Big)\log\Big(\sum_k p_ky_{kj}\Big)$$
	$$\ge \sum_j \Big(\sum_{i>n}p_iy_{ij}\Big)\log \Big(\sum_{i>n}p_iy_{ij}\Big)-\sum_j\sum_{i>n}p_iy_{ij}\log (p_iy_{ij})$$
	$$\ge \sum_j \Big(\sum_{i>n}p_iy_{ij}\Big)\log \Big(\sum_{i>n}p_iy_{ij}\Big)-\sum_j \sum_{i>n}p_iy_{ij}\log \Big(\sum_{l>n}p_ly_{lj}\Big)=0.$$
	
	Since 
	$$\text{MI}(X_1,X^*_t)\ge \text{MI}(X_1,X_t)\ge C_3,$$ 	
	$$-\sum_{l>n}p_l\log p_l<C_3/3,$$  
	$$-\Big(\sum_{l\le n}p_l\Big)\log\Big(\sum_{l\le n}p_l\Big)<C_3/3,$$ 
	we have 
	$$\MI(\ddot{X}_1,\ddot{X}_t)\ge C_3/3.$$
	\qed
\end{proof}

\begin{proof}[Lemma \ref{l5}]
	Define $q_i=p_i/\sum_{j\le n}p_j$, $q_i'=p_i'/\sum_{j\le n}p_j'$. Let $\mu_i$ be the distribution that takes $s_j$ with probability $y_{ij}$. $\sum_{i\le n}q_i\mu_i$ is a distribution that takes $s_j$ with probability $\sum_{i\le n}q_iy_{ij}$.
	
	$$\MI(\ddot{X}_1,\ddot{X}_t)=-\sum_{i\le n}q_i \log q_i +\text{H}\Big(\sum_{i\le n}q_i\mu_i\Big)+\sum_{i\le n}\sum_jq_iy_{ij} \log (q_iy_{ij})$$
	$$=\text{H}\Big(\sum_{i\le n}q_i\mu_i\Big)-\sum_{i\le n}q_i\text{H}(\mu_i)$$
	Since $p_i$ and $p_i'$ are positive, we could define $C_4=\max_{i\le n}\{q_i/q_i'\}$, which is finite and positive.
	
	Since the entropy function $\text{H}$ is concave down, we have Jensen's inequality $\sum_{i\le n}r_i\text{H}(\mu_i)\le \text{H}(\sum_{i\le n}r_i\mu_i)$ for $r_i\ge 0$ with $\sum_{i\le n}r_i=1$.
	
	Thus
	$$\frac{1}{C_4}\text{H}\Big(\sum_{i\le n}q_i\mu_i\Big)+\sum_{i\le n}(q_i'-q_i/C_4)\text{H}(\mu_i)\le \text{H}\Big(\sum_{i\le n}q_i'\mu_i\Big),$$
	
	from which we have
	$$\MI(\ddot{X}_1,\ddot{X}_t)=\text{H}\Big(\sum_{i\le n}q_i\mu_i\Big)-\sum_{i\le n}q_i\text{H}(\mu_i)$$
	$$\le C_4\text{H}\Big(\sum_{i\le n}q_i'\mu_i\Big)-C_4\sum_{i\le n}q_i'\text{H}(\mu_i)=C_4\MI(\tilde{X}_1,\tilde{X}_t).$$
	\qed
\end{proof}

\begin{proof}[Lemma \ref{l6}]
	Construct $\hat{X}_1$ and $\hat{X}_t$ with joint distribution $$\mathbb{P}(\hat{X}_1=s_i,\hat{X}_t=s_j)=y_{ij}p_i'/\sum_{l\ge n}p_l'$$ for $i\ge n$ and any $j$.
	
	$$\Big[\sum_{i\le n}p_i'\Big]\MI(\tilde{X}_1,\tilde{X}_t)+\Big[\sum_{i>n}p_i'\Big]\MI(\hat{X}_1,\hat{X}_t)$$
	$$=\Big[\sum_{i\le n}p_i'\Big]\text{H}\left[\frac{\sum_{i\le n}p_i'\mu_i}{\sum_{i\le n}p_i'}\right]-\sum_{i\le n}p_i'\text{H}(\mu_i)+\Big[\sum_{i>n}p_i'\Big]\text{H}\left[\frac{\sum_{i> n}p_i'\mu_i}{\sum_{i> n}p_i'}\right]-\sum_{i> n}p_i'\text{H}(\mu_i)$$
	$$\le \text{H}\Big[\sum_{i=1}^\infty p_i'\mu_i\Big]-\sum_{i=1}^\infty p_i'\text{H}(\mu_i)=\MI^b(X_1,X_t).$$
	Choose $C_5=\sum_{i\le n}p_i'>0$, then $\MI^b(X_1,X_t)\ge C_5\MI(\tilde{X}_1,\tilde{X}_t)$.
	\qed
\end{proof}

\section{Lifting and Thermodynamic Quantities of Multidimensional Diffusion Processes}
\label{sec-n-torus}

\subsection{Diffusion processes on Euclidean space and torus}

Consider a time-homogeneous diffusion process $\bm{X}(t)$ on $\mathbb{R}^n$:

$$\mathrm{d} \bm{X}(t)=\bm{\Gamma}(\bm{X})\mathrm{d}\bm{B} (t)+\bm{b}(\bm{X})\mathrm{d}t,$$
where $\bm{B}(t)$ is an $n$-dimensional standard Brownian motion. The drift parameter $\bm{b}(\bm{x})$ is $\mathbb{R}^n\to \mathbb{R}^n$, $C^\infty$, with period $1$ for each component. The diffusion parameter $\bm{\Gamma}(\bm{x})$ is $\mathbb{R}^n\to \mathbb{R}^{n\times n}$,  non-degenerate for each $\bm{x}$, $C^\infty$, and $1$-periodic for each component.   We shall
also denote $\bm{D}(\bm{x})=\frac{1}{2}\bm{\Gamma}(\bm{x})\bm{\Gamma}^T(\bm{x})$.  It is positive definite for
each $\bm{x}$.  All vectors are $n\times 1$.

The transition probability density function $f\big(\bm{x},t|\bm{x}_0\big)$ of the diffusion process $\bm{X}(t)$ is the
fundamental solution to the linear, Kolmogorov
forward equation:
$$
\frac{\partial f(\bm{x},t)}{\partial t}=-\nabla\cdot\Big[\bm{b}(\bm{x})f(\bm{x},t)\Big]+\nabla\cdot \nabla\cdot \left[\bm{D}(\bm{x})f(\bm{x},t)\right].
$$
For an $n\times n$ matrix $\bm{M}$ with $i$-th row $\bm{M}_i$, $\nabla\cdot \bm{M}$ is defined as $n\times 1$ vector $(\nabla\cdot \bm{M}_1,\cdots,\nabla\cdot \bm{M}_n)^T$.

In parallel, consider a time-homogeneous diffusion process $\bar{\bm{X}}(t)$ on $\mathbb{T}^n$, where $\mathbb{T}^n$ is defined as $[0,1)^n$:
$$\mathrm{d} \bar{\bm{X}}(t)=\bar{\bm{\Gamma}}(\bar{\bm{X}})\mathrm{d}\bar{\bm{B}} (t)+\bar{\bm{b}}(\bar{\bm{X}})\mathrm{d}t.$$
Here $\bar{\bm{B}}(t)$ is an $n$-dimensional standard Brownian motion on $\mathbb{T}^n$. $\bar{\bm{\Gamma}}(\cdot)$ and $\bar{\bm{b}}(\cdot)$ are the restrictions of $\bm{\Gamma}(\cdot)$ and $\bm{b}(\cdot)$ on $\mathbb{T}^n$.  Similarly, the transition probability 
density function for $\bar{\bm{X}}(t)$, $\bar{f}(\bar{\bm{x}},t|\bar{\bm{x}}_0)$ satisfies the Kolmogorov forward equation:
$$
\frac{\partial \bar{f}(\bar{\bm{x}},t)}{\partial t}=-\nabla\cdot \Big[\bar{\bm{b}}(\bar{\bm{x}})\bar{f}(\bar{\bm{x}},t)\Big]+\nabla\cdot \nabla\cdot \left[ \bar{\bm{D}}(\bar{\bm{x}})\bar{f}(\bar{\bm{x}},t)\right],
$$
in which $\bar{\bm{D}}(\cdot)$ is the restriction of $\bm{D}(\cdot)$ on $\mathbb{T}^n$.

For the above diffusion process $\bm{X}(t)$ on $\mathbb{R}^n$ with periodic diffusion and drift, we can fold it to $\mathbb{T}^n$ by  $$\bar{\bm{X}}(t)=\bm{X}(t) \mod 1,$$
where $\mod 1$ is for every component. The folding result is exactly the above diffusion process $\bar{\bm{X}}(t)$ on $\mathbb{T}^n$ with density function
$$
\bar{f}(\bar{\bm{x}},t)=\sum_{i_1=-\infty}^{+\infty}\cdots\sum_{i_n=-\infty}^{+\infty}
f(\bar{\bm{x}}
+ i_1 \bm{e}_1+\cdots+ i_n \bm{e}_n,t\big),
$$
where $\bm{e}_k$ is an elementary $n$-vector, with $1$ as its $k$-th component and $0$ for other components.

\subsection{Stationary distributions and measures}

The diffusion process $\bar{\bm{X}}(t)$ on $\mathbb{T}^n$ has a stationary distribution $\bar{\rho}(\bar{\bm{x}})$. Its $1$-periodic continuation to $\mathbb{R}^n$,
$$\rho(\bm{x})=\bar{\rho}(\bm{x}\hspace{-0.25cm}\mod 1),$$
is a stationary measure of the diffusion process $\bm{X}(t)$ on $\mathbb{R}^n$.

To further study the stationary distributions and measures, we need to consider the relative entropy of $f(\bm{x},t)$ with respect to any stationary measure $\nu(\bm{x})$
$$\text{D}_{\text{KL}}[f(t),\nu]=\int_{\mathbb{R}^n}f(\bm{x},t)\log\frac{f(\bm{x},t)}{\nu(\bm{x})}\mathrm{d}\bm{x}.$$

In this paper, we assume that $f(t)$ and its spatial derivatives decay fast enough, such that related integrations on $\partial \mathbb{R}^n$ are all $0$.

\begin{lemma}
	\label{dfm}
	$\text{D}_{\text{KL}}[f(t),\nu]$ is monotonically decreasing with $t$.
\end{lemma}
\begin{proof}
	$$\text{D}_{\text{KL}}[f(t),\nu]=\frac{\mathrm{d}}{\mathrm{d}t}\int_{\mathbb{R}^n}f\log\frac{f}{\nu}\mathrm{d}\bm{x}=\int_{\mathbb{R}^n}\frac{\partial f}{\partial t}\log\frac{f}{\nu}\mathrm{d}\bm{x}+\int_{\mathbb{R}^n}\frac{\partial f}{\partial t}\mathrm{d}\bm{x}$$
	$$=\int_{\mathbb{R}^n}\left[-\nabla\cdot (\bm{b}f)+\nabla\cdot \nabla\cdot (\bm{D}f)\right]\log\frac{f}{\nu}\mathrm{d}\bm{x}$$
	$$=\int_{\partial\mathbb{R}^n}\nabla\cdot\left[- \bm{b}f\log\frac{f}{\nu}+ \nabla\cdot \left(\bm{D}f\right)\log\frac{f}{\nu}\right]\mathrm{d}S-\int_{\mathbb{R}^n}\left[- \bm{b}f+ \nabla\cdot \left(\bm{D}f\right)\right]\cdot\nabla\left(\log\frac{f}{\nu}\right)\mathrm{d}\bm{x}$$
	$$=-\int_{\mathbb{R}^n}\left[- \bm{b}f+f \nabla\cdot \bm{D}+\bm{D}\nabla f\right]\cdot\left(\frac{\nu}{f} \nabla\frac{f}{\nu}\right)\mathrm{d}\bm{x}$$
	$$=-\int_{\mathbb{R}^n}\left[- \bm{b}\nu+ \nu\nabla\cdot \bm{D}+\bm{D}\frac{\nu}{f}\nabla f\right]\cdot\nabla\left(\frac{f}{\nu}\right)\mathrm{d}\bm{x}$$
	$$=-\int_{\mathbb{R}^n}\left[- \bm{b}\nu+ \nabla\cdot \left(\bm{D}\nu\right)-\bm{D}\nabla \nu+\bm{D}\frac{\nu}{f}\nabla f\right]\cdot\nabla\left(\frac{f}{\nu}\right)\mathrm{d}\bm{x}$$
	$$=-\int_{\partial\mathbb{R}^n}\frac{f}{\nu}[\nabla\cdot(\bm{D}\nu)-\bm{b}\nu]\mathrm{d}S+\int_{\mathbb{R}^n}\frac{f}{\nu}\nabla\cdot\left[- \bm{b}\nu+ \nabla\cdot \left(\bm{D}\nu\right)\right]\mathrm{d}\bm{x}$$
	$$-\int_{\mathbb{R}^n}\left[\bm{D} \left(\frac{\nu}{f}\nabla f-\nabla \nu\right)\right]\cdot\nabla\left(\frac{f}{\nu}\right)\mathrm{d}\bm{x}$$
	$$=-\int_{\partial\mathbb{R}^n}\nabla\cdot\left[\bm{D} \left(\frac{\nu}{f}\nabla f-\nabla \nu\right)\frac{f}{\nu}\right]\mathrm{d}S+\int_{\mathbb{R}^n}\frac{f}{\nu}\nabla\cdot\left[\bm{D} \left(\frac{\nu}{f}\nabla f-\nabla \nu\right)\right]\mathrm{d}\bm{x}$$
	$$=\int_{\mathbb{R}^n}-\frac{f}{\nu}\nabla\cdot\left(\bm{D} f \nabla\frac{\nu}{f}\right)\mathrm{d}\bm{x}$$
	$$=\int_{\partial\mathbb{R}^n}-\nabla\cdot\left(\frac{f}{\nu} \bm{D} f \nabla\frac{\nu}{f}\right)\mathrm{d}S+\int_{\mathbb{R}^n}\left(\bm{D} f \nabla\frac{\nu}{f}\right)\cdot\nabla\frac{f}{\nu}\mathrm{d}\bm{x}$$
	$$=-\int_{\mathbb{R}^n}\left(f\nabla \nu-\nu\nabla f\right)^T\frac{\bm{D}}{f\nu^2} \left(f\nabla \nu-\nu\nabla f\right)\mathrm{d}\bm{x},$$
	which is non-positive. It is $0$ if and only if $f\nabla \nu=\nu\nabla f$, namely $\nabla \log f=\nabla \log \nu$, thus $f=c\nu$ for a constant $c$.
	\qed
\end{proof}

The above lemma is also valid for diffusion on torus, thus $\text{D}_{\text{KL}}[\bar{f}(t),\bar{\rho}]$ is monotonically decreasing for any initial distribution $\bar{f}(0)$. If the diffusion on torus has another stationary distribution $\bar{\theta}$, then $\text{D}_{\text{KL}}[\bar{\theta},\bar{\rho}]$ is a constant. However it should decrease unless $\bar{\rho}=c\bar{\theta}$, which means $\bar{\theta}=\bar{\rho}$. Therefore, the diffusion on torus has a unique stationary distribution, and any initial distribution will converge to it.

The lifted diffusion on $\mathbb{R}^n$ has no stationary distribution. An intuition is that $\mathbb{R}^n$ is not compact, and the density function $f(\bm{x},t)$ will converge to $0$ at each $\bm{x}$ as $t\to \infty$.
\begin{proposition}
	\label{nsd}
	The lifted diffusion process has no stationary probability distribution.
\end{proposition}
\begin{proof}
	Assume there is a stationary probability distribution $p$. Let $f(t)=p$, then $\text{D}_{\text{KL}}(p,\rho)$ is a constant. This is true only if the equality holds in Lemma \ref{dfm}, which means $p$ and $\rho$ only differ by a constant multiple. However $\rho$ is non-normalizable, so is $f$.
	\qed
\end{proof}

Different from lifted Markov chain, the detailed balanced stationary measure does not always exist for lifted diffusion.

If the detailed balanced stationary measure exists, then the probability flux satisfies 
$$\bm{J}=\nabla\cdot(\bm{D}f)-\bm{b}f=\bm{0}.$$
This is equivalent to 
$$-\bm{D}^{-1}(\nabla\cdot \bm{D}-\bm{b})=\nabla\log f.$$

In general, this is impossible, since $\bm{D}^{-1}(\nabla\cdot \bm{D}-\bm{b})$ may not be curl-free (conservative). 

Global potential is incompatible with asymmetrical cycles. The idea of lifting is to expand asymmetric cycles. For Markov chains, since cycle number is finite, we can expand all of them, such that in the lifted Markov chain, there is no asymmetric cycle. For diffusion process on $\mathbb{T}^n$, there are $n$ topologically non-trivial basic cycles, which might be asymmetric. We expand these cycles and lift the process into $\mathbb{R}^n$. However, there are still infinite many local cycles in the folded process, which are homotopic to a single point. In such local cycles, the curl is not always $0$, therefore these cycles might be asymmetric, and we cannot expand all of them \cite{zdw}. 

Assume that $-\bm{D}^{-1}(\nabla\cdot \bm{D}-\bm{b})=\nabla g(\bm{x})$ is curl-free in $\mathbb{R}^n$. Then there is a detailed balanced stationary measure, $\mu(\bm{x})=ce^{g(\bm{x})}$, where $c$ is any positive number. Then $\bm{J}=\nabla\cdot(\bm{D}\mu)-\bm{b}\mu=\bm{0}$.

\subsection{Instantaneous entropy production rate, free energy, and housekeeping heat}
For the lifted diffusion process $\bm{X}(t)$ on $\mathbb{R}^n$ with probability density function $f(\bm{x},t)$, one could define several thermodynamics quantities: entropy production rate, free energy and housekeeping heat.

\begin{definition}
	The instantaneous free energy with respect to stationary measure $\nu$, $F^\nu(t)$, is defined as $F^\nu(t)=\text{D}_{\text{KL}}(f,\nu)$. Its time derivative is
	$$\mathrm{d}F^\nu(t)/\mathrm{d}t=-\int_{\mathbb{R}^n}\big[- \bm{b}f+f \nabla\cdot \bm{D}+\bm{D}\nabla f\big]\cdot\left[\frac{\nabla f}{f} -\frac{\nabla \nu}{\nu}\right]\mathrm{d}\bm{x}.$$
\end{definition}
\begin{definition}
	The instantaneous entropy production rate $e_p(t)$ is defined as \cite{JQQ}
	$$e_p(t)=\int_{\mathbb{R}^n} (-\bm{b}f+f\nabla\cdot \bm{D}+\bm{D}\nabla f)^Tf^{-1}\bm{D}^{-1}(-\bm{b}f+f\nabla\cdot \bm{D}+\bm{D}\nabla f)\mathrm{d}\bm{x}$$
	$$=\int_{\mathbb{R}^n}\big[- \bm{b}f+f \nabla\cdot \bm{D}+\bm{D}\nabla f\big]\cdot\left[\frac{\nabla f}{f} +\bm{D}^{-1}\nabla \cdot \bm{D}-\bm{D}^{-1}\bm{b}\right]\mathrm{d}\bm{x}.$$
\end{definition}
\begin{definition}
	The instantaneous housekeeping heat with respect to stationary measure $\nu$, $Q_{hk}^\nu(t)$, is defined as 
	$$Q_{hk}^\nu(t)=e_p(t)+\mathrm{d}F^\nu(t)/\mathrm{d}t=\int_{\mathbb{R}^n}\big[- \bm{b}f+f \nabla\cdot \bm{D}+\bm{D}\nabla f\big]\cdot\left[\frac{\nabla \nu}{\nu} +\bm{D}^{-1}\nabla \cdot \bm{D}-\bm{D}^{-1}\bm{b}\right]\mathrm{d}\bm{x}.$$
\end{definition}

These quantities can be also defined for the diffusion process $\bar{\bm X}(t)$ on $\mathbb{T}^n$, denoted as $\bar{F}(t), \bar{e_p}, \bar{Q}_{hk}$. One just needs to replace $\bm{b}, \bm{D}, f, \nu, \mathbb{R}^n$ by $\bar{\bm{b}}, \bar{\bm{D}}, \bar{f}, \bar{\rho}, \mathbb{T}^n$.

From Lemma \ref{dfm}, $\mathrm{d}F^\nu(t)/\mathrm{d}t\le 0$. Since $\bm{D}$ is positive definite, $e_p(t)\ge 0$. For $Q^\nu_{hk}(t)$, we have the same result.
\begin{proposition}
	$Q_{hk}^\nu(t)\ge 0$.
\end{proposition}
\begin{proof}
	$$Q_{hk}^\nu(t)=\int_{\mathbb{R}^n}\big[- \bm{b}f+f \nabla\cdot \bm{D}+\bm{D}\nabla f\big]\cdot\left[\frac{\nabla \nu}{\nu} +\bm{D}^{-1}\nabla \cdot \bm{D}-\bm{D}^{-1}\bm{b}\right]\mathrm{d}\bm{x}$$
	$$=\int_{\mathbb{R}^n}f\left[\frac{\nabla \nu}{\nu} +\bm{D}^{-1}\nabla \cdot \bm{D}-\bm{D}^{-1}\bm{b}\right]^T \bm{D} \left[- \bm{D}^{-1}\bm{b}+\bm{D}^{-1}\nabla\cdot \bm{D}+\frac{\nabla f}{f}\right]\mathrm{d}\bm{x}$$
	$$=\int_{\mathbb{R}^n}f\left[\frac{\nabla \nu}{\nu} +\bm{D}^{-1}\nabla \cdot \bm{D}-\bm{D}^{-1}\bm{b}\right]^T \bm{D} \left[- \bm{D}^{-1}\bm{b}+\bm{D}^{-1}\nabla\cdot \bm{D}+\frac{\nabla \nu}{\nu}\right]\mathrm{d}\bm{x}$$
	$$+\int_{\mathbb{R}^n}f\left[\frac{\nabla \nu}{\nu} +\bm{D}^{-1}\nabla \cdot \bm{D}-\bm{D}^{-1}\bm{b}\right]^T \bm{D} \left[- \frac{\nabla \nu}{\nu}+\frac{\nabla f}{f}\right]\mathrm{d}\bm{x}.$$
	Since $\bm{D}$ is positive definite, the first term is non-negative. The second term equals 
	$$\int_{\mathbb{R}^n}[\bm{D}\nabla \nu +\nu\nabla \cdot \bm{D}-\bm{b}\nu]^T \frac{f}{\nu} \left[- \frac{\nabla \nu}{\nu}+\frac{\nabla f}{f}\right]\mathrm{d}\bm{x}$$
	$$=\int_{\mathbb{R}^n}[\bm{D}\nabla \nu +\nu\nabla \cdot \bm{D}-\bm{b}\nu]^T \nabla\left(\frac{f}{\nu}\right)\mathrm{d}\bm{x}$$
	$$=-\int_{\mathbb{R}^n}\frac{f}{\nu}\nabla\cdot[\bm{D}\nabla \nu +\nu\nabla \cdot \bm{D}-\bm{b}\nu] \mathrm{d}\bm{x}=0,$$
	since $\nu$ is a stationary measure, $\nabla\cdot\nabla\cdot(\bm{D}\nu)-\nabla\cdot (\bm{b}\nu)=0$.
	\qed
\end{proof}
Thus we have the decomposition
$$e_p(t)=Q_{hk}^\nu(t)+[-\mathrm{d}F^\nu(t)/\mathrm{d}t],$$
where each term is non-negative. This is also valid for the torus version.

\subsection{Time limits of thermodynamic quantities}
Since $\bar{f}(t)$ converges to $\bar{\rho}$, $\bar{F}(t)$ and $\mathrm{d}\bar{F}(t)/\mathrm{d}t$ converge to $0$, $\bar{e_p}(t)$ and $\bar{Q}_{hk}(t)$ converge to the stationary entropy production rate 
$$\bar{e_p}=\int_{\mathbb{T}^n}\frac{1}{\bar{\rho}}[- \bar{\bm{b}}\bar{\rho}+\bar{\rho} \nabla\cdot \bar{\bm{D}}+\bar{\bm{D}}\nabla \bar{\rho}]^T\bar{\bm{D}}^{-1}[- \bar{\bm{b}}\bar{\rho}+\bar{\rho} \nabla\cdot \bar{\bm{D}}+\bar{\bm{D}}\nabla \bar{\rho}]\mathrm{d}\bar{\bm{x}}.$$

For the lifted diffusion process, $f(t)$ does not converge to a stationary distribution, therefore we do not have the stationary version of these quantities. However, we can still study their behavior as $t\to \infty$.

If we set $\nu$ to be the periodic stationary measure $\rho$, then $Q_{hk}^\rho(t)$ converges to
$$\int_{\mathbb{T}^n}[- \bar{\bm{b}}\bar{\rho}+\bar{\rho} \nabla\cdot \bar{\bm{D}}+\bar{\bm{D}}\nabla \bar{\rho}]\cdot\left[\frac{\nabla \bar{\rho}}{\bar{\rho}} +\bar{\bm{D}}^{-1}\nabla \cdot \bar{\bm{D}}-\bar{\bm{D}}^{-1}\bar{\bm{b}}\right]\mathrm{d}\bar{\bm{x}},$$
which is just $\bar{e_p}$.

If we set $\nu$ to be the detailed balanced stationary measure $\mu$ (if exists), then $Q_{hk}^\mu(t)\equiv 0$ since $\nabla\cdot(\bm{D}\mu)-\bm{b}\mu=\bm{0}$.

For the time limit of $e_p(t)$, we have the following theorem. The proof is in the next part.
\begin{theorem}
	\label{thm1}
	For any initial distribution $f(\bm{x},0)$ that has a finite covariance matrix, the entropy production rate of diffusion process ${\bm X}(t)$ 
	on $\mathbb{R}^n$, $e_p(t)$, also converges to $\bar{e_p}$ in Ces\`{a}ro's sense that
	\[
	\lim_{T\to\infty}\frac{1}{T}\int_0^T e_p(t)\mathrm{d}t=\bar{e_p}.
	\]
\end{theorem}

In summary, $e_p=Q^\nu_{hk}+(-\mathrm{d}F^\nu/\mathrm{d}t)$, where $e_p$, $Q^\nu_{hk}$ and $-\mathrm{d}F^\nu/\mathrm{d}t$ are non-negative.

$e_p\to\bar{e_p}$ in Ces\`{a}ro's sense, $\mathrm{d}F^\mu/\mathrm{d}t=-e_p\to -\bar{e_p}$ in Ces\`{a}ro's sense, $Q_{hk}^\mu\equiv 0$, $\mathrm{d}F^\rho/\mathrm{d}t\to 0$ in Ces\`{a}ro's sense, $Q_{hk}^\rho\to \bar{e_p}$ in general sense.

Therefore, the periodic stationary measure $\rho$ and the detailed balanced stationary measure $\mu$ reach the maximum and minimum of $Q_{hk}^\nu$, $\bar{e_p}$ and $0$, as $t\to \infty$.

When $-\bm{D}^{-1}(\nabla\cdot \bm{D}-\bm{b})$ is not curl-free, $\mu$ does not exist, and the minimum of $Q_{hk}^\nu$ is larger than $0$.

\subsection{Proof of Theorem \ref{thm1}}
For a distribution $q(\bm{x})$, its entropy is defined as
$$\text{H}[q]=\int_{\mathbb{R}^n}-q(\bm{x})\log q(\bm{x})\mathrm{d}\bm{x}.$$

We have a famous result that the maximal entropy under fixed variance is achieved by normal distributions \cite{max}:
\begin{lemma}
	\label{var}
	For a continuous probability density function $p$ on $\mathbb{R}^n$ with fixed covariance matrix $\Sigma$, its entropy $\text{H}[p]$ satisfies
	$$\text{H}[p]\le \frac{1}{2}\Big[n+\log \big(2^n\pi^n \det \Sigma\big)\Big].$$
	The equality holds if and only if $p$ is an $n$-dimensional normal distribution with covariance matrix $\Sigma$.
\end{lemma}

The proof consists of the following lemmas. The key idea is to transform the convergence of entropy production rate into the control of entropy.

\begin{lemma}
	\label{le4}
	Assume the initial distribution $f(\bm{x},0)$ has finite covariance matrix. Then 
	$$\lim_{T\to\infty}\frac{1}{T}\int_0^T e_p(t)\mathrm{d}t-\bar{e_p}=0\Longleftrightarrow \text{H}[f(T)]/T\to 0.$$
\end{lemma}
\begin{proof}
	We have
	$$e_p(t)-\bar{e_p}(t)=\int_{\mathbb{R}^n}\frac{1}{f}(\nabla f)^T\bm{D}\nabla f  \mathrm{d}\bm{x}-\int_{\mathbb{T}^n}\frac{1}{\bar{f}}(\nabla \bar{f})^T\bar{\bm{D}}\nabla \bar{f}  \mathrm{d}\bar{\bm{x}}=-\frac{\mathrm{d}F^\rho(t)}{\mathrm{d}t}+\frac{\mathrm{d}\bar{F}(t)}{\mathrm{d}t}.$$
	
	Thus
	$$\frac{1}{T}\int_0^T e_p(t)\mathrm{d}t-\frac{1}{T}\int_0^T \bar{e_p}(t)\mathrm{d}t=\frac{1}{T}[F^\rho(0)-F^\rho(t)+\bar{F}(T)-\bar{F}(0)].$$
	
	Since $\bar{f}(\bar{\bm{x}},t)$ converges to $\bar{\rho}(\bar{\bm{x}})$, $\bar{F}(t)$ converges to $0$. Therefore $\bar{F}(T)$ is bounded.
	
	$F^\rho(0)-\bar{F}(0)=\text{H}[\bar{f}(0)]-\text{H}[f(0)]$. Since $f(\bm{x},0)$ has finite covariance matrix, Lemma \ref{var} shows that $\text{H}[f(0)]$ is finite, so as $\text{H}[\bar{f}(0)]$. 
	
	We also have $\lim_{T\to\infty}\frac{1}{T}\int_0^T \bar{e_p}(t)\mathrm{d}t=\bar{e_p}.$
	
	Since $F^\rho(T)=-\text{H}[f(T)]-\int_{\mathbb{R}^n}f(\bm{x},T)\log \rho\mathrm{d}\bm{x}$, and $\int_{\mathbb{R}^n}f(\bm{x},T)\log \rho\mathrm{d}\bm{x}$ converges to 
	\: $\int_{\mathbb{T}^n}\rho(\bar{\bm{x}})\log \rho(\bar{\bm{x}})\mathrm{d}\bar{\bm{x}}$, which is finite, $F^\rho(T)/T\to 0$ is equivalent to $\text{H}[f(T)]/T\to 0$.
	
	Therefore $\lim_{T\to\infty}\frac{1}{T}\int_0^T e_p(t)\mathrm{d}t-\bar{e_p}=\lim_{T\to\infty}\text{H}[f(T)]/T$.

	\qed
	
\end{proof}

From this proof, we can see that $e_p(t)\to \bar{e_p}\Longleftrightarrow \mathrm{d}\text{H}[f(t)]/\mathrm{d}t\to 0$. The techniques we use can prove $\text{H}[f(t)]/t\to 0$, but not $\mathrm{d}\text{H}[f(t)]/\mathrm{d}t\to 0$. Thus we do not have $e_p(t)\to \bar{e_p}$.

Since $f(\bm{x},t)\le \bar{f}(\bar{\bm{x}},t)$, $\bar{f}(\bar{\bm{x}},t)$ converges to $\bar{\rho}(\bar{\bm{x}})$, $f(\bm{x},t)$ has a uniform upper bound for large $t$. Therefore $\text{H}[p(t)]$ has a finite lower bound for large $t$. We only need to control $\text{H}[f(t)]$ from above. From Lemma \ref{var}, we need to control the covariance matrix of the diffusion process.

\begin{lemma}
	\label{difv}
	Consider the diffusion process $\bm{X}(t)$ on $\mathbb{R}^n$ with initial distribution $f(\bm{x},0)$. Assume the initial distribution $f(\bm{x},t)$ has finite covariance matrix. Then one has constants $C,T_0$ such that for any $T>T_0$, $i,j=1,\cdots,n$, $|\text{Cov}[\bm{X}(T)]_{ij}|\le CT^2$.
\end{lemma}
\begin{proof}
	For the diffusion process $\bm{X}(t)$ on $\mathbb{R}^n$, we have the infinitesimal mean
	$$\mathbb{E}[\bm{X}(t+\Delta t)-\bm{X}(t)\mid\bm{X}(t)=\bm{x}_0]=\bm{b}(\bm{x}_0)\Delta t+O(\Delta t^2).$$
	
	We also have the infinitesimal variance
	$$\mathbb{E}\{[\bm{X}(t+\Delta t)-\bm{X}(t)][\bm{X}(t+\Delta t)-\bm{X}(t)]^T\mid\bm{X}(t)=\bm{x}_0\}=\bm{\Gamma}(\bm{x}_0)\bm{\Gamma}(\bm{x}_0)^T\Delta t+O(\Delta t^2).$$
	
	Therefore the covariance matrix satisfies 
	$$\text{Cov}[\bm{X}(t+\Delta t)\mid\bm{X}(t)=\bm{x}_0]=\bm{\Gamma}(\bm{x}_0)\bm{\Gamma}(\bm{x}_0)^T\Delta t+O(\Delta t^2).$$
	
	Set $G=\max_{\bm{x},i}[\Gamma(\bm{x})\Gamma(\bm{x})^T]_{ii}<\infty$. Then we can choose a small enough $\Delta t$ such that for any $0\le \Delta t'\le \Delta t$, $i=1,\cdots,n$
	$$\text{Var}[\bm{X}_i(t+\Delta t')-\bm{X}_i(t)]\le 2G\Delta t,$$
	regardless of the value of $\bm{X}(t)$.
	
	Denote $D=\max_{i}\text{Var}[\bm{X}_i(0)]$.
	
	For a fixed $T>0$, set $m=\lceil T/\Delta t\rceil$. Then 
	$$\text{Cov}[\bm{X}(T)]=\text{Cov}\{\bm{X}(0)+[\bm{X}(\Delta t)-\bm{X}(0)]+\cdots +[\bm{X}(T)-\bm{X}((m-1)\Delta t)]\}.$$
	
	For two random variables $Y,Z$, we have $|\text{Cov}(Y,Z)|\le \sqrt{\text{Var}[Y]\text{Var}[Z]}$.
	
	Applying this inequality to $\text{Cov}[\bm{X}(T)]$, we have 
	$$|\text{Cov}[\bm{X}(T)]_{ij}|\le D+2m\sqrt{2DG\Delta t}+2m^2G\Delta t$$
	
	When $T$ is large enough, $|\text{Cov}[\bm{X}(T)]_{ij}|\le 3(T/\Delta t)^2G\Delta t$.
	\qed
\end{proof}

When $|\text{Cov}[\bm{X}(T)]_{ij}|\le CT^2$, $|\det \text{Cov}[\bm{X}(T)]|\le n! C^n T^{2n}$. Now we have 
\begin{lemma}
	For the diffusion process $\bm{X}(t)$ on $\mathbb{R}^n$, assume the initial distribution $f(\bm{x},0)$ has finite covariance matrix. Then its entropy at time $T$, $\text{H}(T)$, is controlled by $C''\le \text{H}(T)\le C' \log T$ for $T$ large enough, where $C'$ and $C''$ are constants. Therefore $\lim_{T\to\infty}\text{H}(T)/T=0$.
\end{lemma}

This finishes the proof of Theorem \ref{thm1}.

\subsection{Entropy production as energy dissipation}
Consider the free energy with detailed balanced stationary measure $\mu=\exp(-\varphi)$ (if exists)
$$F^\mu(t)=\int_{\mathbb{R}^n}f(\bm{x},t)\log\frac{f(\bm{x},t)}{\mu(\bm{x})}\mathrm{d}\bm{x}=\int_{\mathbb{R}^n}f(\bm{x},t)\log f(\bm{x},t)\mathrm{d}\bm{x}-\int_{\mathbb{R}^n}f(\bm{x},t)\log \mu(\bm{x})\mathrm{d}\bm{x}.$$

Here $S(t)=-\int_{\mathbb{R}^n}f(\bm{x},t)\log f(\bm{x},t)\mathrm{d}\bm{x}$ is the entropy, and $E(t)=\int_{\mathbb{R}^n}f(\bm{x},t)\varphi(\bm{x})\mathrm{d}\bm{x}$ is the mean potential energy. Thus $F^\mu(t)=E(t)-S(t)$.

For $E(t)$, we have the following result:
\begin{proposition}
	The time derivative of $E(t)$ converges to the negative stationary entropy production rate,
	$$\frac{\mathrm{d}E(t)}{\mathrm{d}t}\to -\bar{e_p}.$$
\end{proposition}
\begin{proof}
	$$\frac{\mathrm{d}E(t)}{\mathrm{d}t}+\bar{e_p}(t)=\int_{\mathbb{R}^n}\left[\nabla\cdot(\bm{D}f)-\bm{b}f\right]\cdot\left(\frac{\nabla \mu}{\mu}\right)\mathrm{d}\bm{x}$$
	$$+\int_{\mathbb{R}^n}\left[\nabla\cdot(\bm{D}f)-\bm{b}f\right]\cdot\left(\frac{\nabla \bar{f}}{\bar{f}}+\bm{D}^{-1}\nabla\cdot\bm{D}-\bm{D}^{-1}\bm{b}\right)\mathrm{d}\bm{x}$$
	$$=\int_{\mathbb{T}^n}\left[\nabla\cdot(\bar{\bm{D}}\bar{f})-\bar{\bm{b}}\bar{f}\right]\cdot\left(\frac{\nabla \bar{f}}{\bar{f}}\right)\mathrm{d}\bm{\bar{x}}=-\int_{\mathbb{T}^n}\nabla\cdot\left[\nabla\cdot(\bar{\bm{D}}\bar{f})-\bar{\bm{b}}\bar{f}\right]\log\bar{f}\mathrm{d}\bm{\bar{x}}$$
	$$=-\int_{\mathbb{T}^n}\nabla\cdot\left[\nabla\cdot(\bar{\bm{D}}\bar{\rho})-\bar{\bm{b}}\bar{\rho}\right]\log\bar{\rho}\mathrm{d}\bm{\bar{x}}=0.$$
	
	In the integral $\int_{\mathbb{R}^n}\cdots \mathrm{d}\bm{x}$, $\bar{f}(\bm{x})$ is the $1$-periodic extension of $\bar{f}(\bar{\bm{x}})$. We also apply the facts that $\nabla\cdot\left[\nabla\cdot(\bar{\bm{D}}\bar{\rho})-\bar{\bm{b}}\bar{\rho}\right]=0$ and $\frac{\nabla \mu}{\mu}+\bm{D}^{-1}\nabla\cdot\bm{D}-\bm{D}^{-1}\bm{b}=\bm{0}$.
	
	\qed
\end{proof}

In the decomposition of free energy $F^\mu(t)=E(t)-S(t)$, The first term is asymptotically linear with $t$, and the second term is sub-linear with $t$ (controlled by $C\log t$).

The entropy production of the diffusion process on $\mathbb{T}^n$, which cannot be described by system status quantities directly, is reflected by the free energy/potential energy dissipation of the lifted diffusion process on $\mathbb{R}^n$.

\section{Discussion and Summary of Conclusions}
\label{sec-6}
\subsection{Concerns on probability}
{In this paper, for processes in continuous space, we only consider diffusion processes. The reason is that diffusion processes have continuous trajectories, which imply a well-defined lifting. Consider processes in continuous space with discontinuous trajectories, such as Cauchy process or many other L\'evy processes on $n$-dimensional torus $\mathbb{T}^n$. At discontinuous points of a trajectory, we cannot uniquely determine its lifting into $\mathbb{R}^n$. For example, if the trajectory jumps from $0$ to $0.5$ on $\mathbb{S}^1$, then we do not know whether it is $0$ to $0.5$ or $0$ to $1.5$ after lifting to $\mathbb{R}$. In other words, there are many processes on $\mathbb{R}^n$ with periodic parameters that can be folded back to the same process on $\mathbb{T}^n$. Not all such lifted processes have similar properties as diffusion processes in our paper.}

{In some Markov chains, states can be divided into fast states (large $q_i$, short staying time) and slow states (small $q_i$, long staying time). When $q_i\to \infty$ with the ratio of $q_{ij}$ kept for fast states, the limit process can be tricky: one way is to decimate these fast states, namely coarse graining }\cite{puglisi2010entropy}; {the other way is to keep these transient states. It is very interesting that two ways are different: such transient states, if not decimated, can contribute to the entropy production. How to combine such limit process with lifting might be an interesting future work.}

\subsection{Gibbs potential}

Relative entropy w.r.t. the Lebesgue measure, e.g. Gibbs-Shannon's entropy, 
is not the most appropriate information
characteristics for system with a nontrivial stationary measure.
This insight has already existed in the work of classical thermodynamics,
where entropy as the ``thermodynamic potential'' of an
isolated system with a priori equal probability is replaced by
the free energy as the proper thermodynamic potential for an
isothermal system.  In Gibbs' theory of chemical thermodynamics,
chemical potential of a chemical special $i$ has actually {\em three} 
parts: $\mu_i = \mu_i^o + k_BT\log x_i= h^o - T s^o+ k_BT\log x_i$.  
%These three terms should be identified, with in the present work, 
%\[
%    -\int_0^x \left(\frac{b(x)}{D(x)}\right)\rd x  
%          + \log\left(\frac{D(x)}{D(0)}\right) + \log f(x),
%\]
%in which the first term is the internal potential energy.

\subsection{The nature of nonequilibrium dissipation}

The nature of ``thermodynamic dissipation'' has long been debated.  
A notion that is generally agreed upon was put forward by Onsager \cite{Ons}, who clearly identified a dissipation with a transport process, with both nonzero {\em thermodynamic force} and {\em thermodynamic flux}. In fact, they are necessarily vanishing simultaneously in a thermodynamic equilibrium with detailed balance.  This gives rise to the reciprocal relation in the linear regime near equilibrium.  Macroscopic transport processes, however, can be classified into two gross types: Those induced by a nonequilibrium initial condition and those driven by an active forcing.  This distinction, we believe, is precisely behind Clausius' and Kelvin's statements of the Second Law of Thermodynamics, as well as behind the irreversibility formulated by Boltzmann and I. Prigogine, respectively.   Still, in the latter case, the precise physical step(s) at which dissipation occurs has generated a wide range of arguments:  It is attributed to the external driving force, to the transport processes themselves inside a system, and to the ``boundary'' where a system in contact with its nonequilibrium environment \cite{FTT}.   Our mathematical theory clearly indicates that all these perspectives are not incorrect, but a more precise, complete notion really is to identify nonequilibrium {\em cycles} which have been independently discovered by T. L. Hill \cite{hill-book} and by Laudauer and Bennett \cite{bennett}, in biochemistry and computation respectively.

A cyclic macroscopic transport process driven by a sustained nonequilibrium environment is of course an idealization of reality:  A battery has to be re-charged, the chemical solution that sustains a chemostat has to be replenished.  In fact, almost all engineering systems do not use a continuously charged energy source, but rather rely on the principle of quasi-stationarity \cite{ge-qian-pre-13}:  A  narrow range of decreasing external driving force is acceptable.  Therefore, by clearly recognizing the source of a driving force,  the 
cycles inside a finite driven system can and should be identified with 
a spontaneous ``downhill'' of an external process, as clearly 
shown in the present paper. {A stationary driving force is an idealization; it is represented by the house-keeping heat in the present work.  It corresponds to an unbounded potential energy function on an infinite space.}

Mathematically, this lifting of a driven system with discrete
state space or continuous $n$-torus state spaces has been rigorously 
established in the present work.  Generalization of this result
to $\mathbb{R}^n$ without local potential is technically challenging, 
but that should not prevent our understanding of the nature of the 
physics of nonequilibrium dissipation.  In fact, we propose 
a modernized, combined Clausius-Kelvin statement: 
\begin{quotation}
	``{\it A mesoscopic engine that works in completing irreversible 
		internal cycles statistically has necessarily an external effect that 
		passes heat from a warmer to a colder body.}''
\end{quotation}

\section*{Acknowledgments}
	Y. W. would like to thank Krzysztof Burdzy, Anna Erschler, Noah Forman, Mikhail Gromov, Yang Lan, Soumik Pal, Xin Sun, Jiajun Tong, Dongyi Wei, Boyu Zhang, Yumeng Zhang, Tianyi Zheng for helpful advice and discussions.
	H. Q. thanks Prof. Ping Ao for illuminating discussions.
	This work was supported in part by NIH grant R01GM109964
	(PI: Sui Huang) and the Olga Jung Wan Endowed Professorship.

\bibliographystyle{vancouver}
\bibliography{entropy}

\begin{thebibliography}{10}

\bibitem{pauli}
Pauli W, Enz CP.
\newblock Thermodynamics and the Kinetic Theory of Gases.
\newblock Courier Corporation; 2000.

\bibitem{jarzynski}
Jarzynski C.
\newblock Equalities and inequalities: Irreversibility and the second law of
  thermodynamics at the nanoscale.
\newblock Annu Rev Condens Matter Phys. 2011;2(1):329-51.

\bibitem{seifert}
Seifert U.
\newblock Stochastic thermodynamics, fluctuation theorems and molecular
  machines.
\newblock Rep Prog Phys. 2012;75(12):126001.

\bibitem{herzfeld}
Griffing V, Herzfeld KF.
\newblock Fundamental Physics of Gases.
\newblock Princeton University Press; 2015.

\bibitem{dorfman}
Dorfman JR.
\newblock An Introduction to Chaos in Nonequilibrium Statistical Mechanics.
\newblock Cambridge University Press; 1999.

\bibitem{villani}
Villani C.
\newblock H-Theorem and beyond: {B}oltzmann's entropy in today's mathematics.
\newblock In: Gallavotti G, Reiter WL, Yngvason J, editors. Boltzmann's Legacy.
  Z\"{u}rich: European Mathematical Society; 2008. p. 129-43.

\bibitem{lebowitz-bergmann}
Bergmann PG, Lebowitz JL.
\newblock New approach to nonequilibrium processes.
\newblock Phys Rev. 1955;99(2):578.

\bibitem{qian-bc-1}
Qian H.
\newblock Chemical reaction kinetic perspective with mesoscopic nonequilibrium
  thermodynamics.
\newblock In: Peletier MA, van Santen RA, Steur E, editors. Complexity Science:
  An Introduction. Singapore: World Scientific; 2019. p. 347-73.

\bibitem{qian-bc-2}
Qian H.
\newblock Stochastic population kinetics and its underlying
  mathematicothermodynamics.
\newblock In: Bianchi A, Hillen T, Lewis M, Yi Y, editors. The Dynamics of
  Biological Systems. Switzerland: Springer Interactional Publishing; 2019. p.
  xxx-xxx.

\bibitem{ge-qian-10}
Ge H, Qian H.
\newblock Physical origins of entropy production, free energy dissipation, and
  their mathematical representations.
\newblock Phys Rev E. 2010;81(5):051133.

\bibitem{vandenbroek-esposito}
Van~den Broeck C, Esposito M.
\newblock Three faces of the second law. {II}. {Fokker}-{Planck} formulation.
\newblock Phys Rev E. 2010;82(1):011144.

\bibitem{esposito-vandenbroek}
Esposito M, Van~den Broeck C.
\newblock Three detailed fluctuation theorems.
\newblock Phys Rev Lett. 2010;104(9):090601.

\bibitem{NESS1}
Zhang XJ, Qian H, Qian M.
\newblock Stochastic theory of nonequilibrium steady states and its
  applications. {Part} {I}.
\newblock Phys Rep. 2012;510(1):1-86.

\bibitem{qian-jmp}
Qian H.
\newblock A decomposition of irreversible diffusion processes without detailed
  balance.
\newblock J Math Phys. 2013;54(5):053302.

\bibitem{ge-qian-16-1}
Ge H, Qian H.
\newblock Mesoscopic kinetic basis of macroscopic chemical thermodynamics: {A}
  mathematical theory.
\newblock Phys Rev E. 2016;94(5):052150.

\bibitem{ge-qian-16-2}
Ge H, Qian H.
\newblock Mathematical formalism of nonequilibrium thermodynamics for nonlinear
  chemical reaction systems with general rate law.
\newblock J Stat Phys. 2017;166(1):190-209.

\bibitem{ross}
Ross J, Berry RS.
\newblock Thermodynamics and Fluctuations Far from Equilibrium.
\newblock Springer; 2008.

\bibitem{luojl}
Luo JL, Van Den~Broeck C, Nicolis G.
\newblock Stability criteria and fluctuations around nonequilibrium states.
\newblock Z Phys B. 1984;56(2):165-70.

\bibitem{beard-qian-plos-1}
Beard DA, Qian H.
\newblock Relationship between thermodynamic driving force and one-way fluxes
  in reversible processes.
\newblock PLOS ONE. 2007;2(1):e144.

\bibitem{Levchenko}
Levchenko VV, Fleming R, Qian H, Beard DA.
\newblock An annotated English translation of `{Kinetics} of stationary
  reactions' [{MI Temkin, Dolk. Akad. Nauk SSSR.} 152, 156 (1963)].
\newblock arXiv preprint arXiv:10012861. 2010.

\bibitem{hill-book}
Hill TL.
\newblock Free Energy Transduction in Biology: The Steady-State Kinetic and
  Thermodynamic Formalism.
\newblock Academic Press; 1977.

\bibitem{JQQ}
Jiang DQ, Qian M, Qian MP.
\newblock Mathematical Theory of Nonequilibrium Steady States: on the Frontier
  of Probability and Dynamical Systems.
\newblock Berlin: Springer-Verlag; 2004.

\bibitem{altaner}
Altaner B.
\newblock Foundations of stochastic thermodynamics.
\newblock arXiv preprint arXiv:14103983. 2014.

\bibitem{qkkb}
Qian H, Kjelstrup S, Kolomeisky AB, Bedeaux D.
\newblock Entropy production in mesoscopic stochastic thermodynamics:
  {Nonequilibrium} kinetic cycles driven by chemical potentials, temperatures,
  and mechanical forces.
\newblock J Phys Condens Matter. 2016;28(15):153004.

\bibitem{Lieb-Yngvason}
Lieb EH, Yngvason J.
\newblock A fresh look at entropy and the second law of thermodynamics.
\newblock Phys Today. 200;53(4):32-8.

\bibitem{planck}
Planck M.
\newblock Treatise on Thermodynamics.
\newblock 3rd ed. Dover; 1910.

\bibitem{ge2012stochastic}
Ge H, Qian M, Qian H.
\newblock Stochastic theory of nonequilibrium steady states. Part II:
  Applications in chemical biophysics.
\newblock Phys Rep. 2012;510(3):87-118.

\bibitem{bennett}
Bennett CH.
\newblock Notes on {Landauer}'s principle, reversible computation, and
  {Maxwell}'s Demon.
\newblock Stud Hist Philos Sci B. 2003;34(3):501-10.

\bibitem{munkres1974topology}
Munkres JR.
\newblock Topology: a First Course.
\newblock Prentice-Hall; 1974.

\bibitem{G173}
Diestel R.
\newblock Graph Theory.
\newblock 2nd ed. New York: Springer-Verlag; 2000.

\bibitem{godsil2013algebraic}
Godsil C, Royle GF.
\newblock Algebraic Graph Theory.
\newblock Springer Science \& Business Media; 2013.

\bibitem{norris1998markov}
Norris JR.
\newblock Markov chains.
\newblock Cambridge university press; 1998.

\bibitem{oksendal2003stochastic}
{\O}ksendal B.
\newblock Stochastic differential equations.
\newblock Springer; 2003.

\bibitem{Pol}
Polettini M.
\newblock Cycle/cocycle oblique projections on oriented graphs.
\newblock Lett in Math Phys. 2015;105(1):89-107.

\bibitem{morimoto}
Morimoto T.
\newblock Markov processes and the {H}-theorem.
\newblock J Phys Soc Jpn. 1963;18(3):328-31.

\bibitem{voigt}
Voigt J.
\newblock Stochastic operators, information, and entropy.
\newblock Commun Math Phys. 1981;81(1):31-8.

\bibitem{thompson-qian}
Thompson LF, Qian H.
\newblock Potential of entropic force in {Markov} systems with nonequilibrium
  steady state, generalized {Gibbs} function and criticality.
\newblock Entropy. 2016;18(8):309.

\bibitem{Kai}
Kaimanovich VA, Woess W.
\newblock Boundary and entropy of space homogeneous {Markov} chains.
\newblock Ann Probab. 2002;30(1):323-63.

\bibitem{erschler2003drift}
Erschler A.
\newblock On drift and entropy growth for random walks on groups.
\newblock Ann Probab. 2003;31(3):1193-204.

\bibitem{saloff2016random}
Saloff-Coste L, Zheng T.
\newblock Random walks and isoperimetric profiles under moment conditions.
\newblock Ann Probab. 2016;44(6):4133-83.

\bibitem{kaimanovich1983random}
Kaimanovich VA, Vershik AM.
\newblock Random walks on discrete groups: boundary and entropy.
\newblock Ann Probab. 1983;11(3):457-90.

\bibitem{DPI}
Cover TM, Thomas JA.
\newblock Elements of Information Theory.
\newblock John Wiley \& Sons; 2012.

\bibitem{zdw}
Qian M, Wang ZD.
\newblock The entropy production of diffusion processes on manifolds and its
  circulation decompositions.
\newblock Commun Math Phys. 1999;206(2):429-45.

\bibitem{max}
Schuster P.
\newblock Stochasticity in Processes.
\newblock Springer; 2016.

\bibitem{puglisi2010entropy}
Puglisi A, Pigolotti S, Rondoni L, Vulpiani A.
\newblock Entropy production and coarse graining in Markov processes.
\newblock J Stat Mech. 2010;2010(05):P05015.

\bibitem{Ons}
Onsager L.
\newblock Reciprocal relations in irreversible processes. {I}.
\newblock Phys Rev. 1931;37(4):405.

\bibitem{FTT}
Andresen B, Salamon P, Berry RS.
\newblock Thermodynamics in finite time.
\newblock Phys Today. 1984;37(9):62-70.

\bibitem{ge-qian-pre-13}
Ge H, Qian H.
\newblock Heat dissipation and nonequilibrium thermodynamics of quasi-steady
  states and open driven steady state.
\newblock Phys Rev E. 2011;87:062125.

\end{thebibliography}

\end{document}